\documentclass[reqno,centertags,11pt]{amsart}

\usepackage{amssymb,amsmath,amsfonts,amssymb}%numbysec}
\usepackage{hyperref}
\usepackage{enumerate}

-\marginparwidth \oddsidemargin -9mm \evensidemargin \oddsidemargin
\usepackage{latexsym}
\advance\hoffset by 5mm

\def\@abssec#1{\vspace{.05in}\footnotesize \parindent .2in
{\bf #1. }\ignorespaces}

\newtheorem{theorem}{Theorem}[section]

\newtheorem{lemma}[theorem]{Lemma}

\newtheorem{remark}[theorem]{Remark}

\allowdisplaybreaks \numberwithin{equation}{section}

\begin{document}

\title[Discretely self-similar solutions for Euler equations]{Discretely self-similar singular solutions for the incompressible Euler equations}
\author{Liutang Xue}
\address{School of Mathematical Sciences, Beijing Normal University and Laboratory of Mathematics and Complex Systems, Ministry of Education, Beijing 100875, P.R. China}
\email{xuelt@bnu.edu.cn}
\date{}
\maketitle

\begin{abstract}
  In this article we consider the discretely self-similar singular solutions of the Euler equations,
and the possible velocity profiles concerned not only have decaying spatial asymptotics, but also have unconventional non-decaying asymptotics.
By relying on the local energy inequality of the velocity profiles and the bootstrapping method,
we prove some nonexistence results and show the energy behavior of the possible nontrivial velocity profiles.
For the case with non-decaying asymptotics, the needed representation formula of the pressure profile in terms of velocity profiles is also given and justified.

\end{abstract}

\section{Introduction}
In this paper we consider the Cauchy problem of the $N$-dimensional ($N\geq 3$) incompressible Euler equations
\begin{equation}\label{eq:euler}
\begin{cases}
  \partial_t v + v\cdot \nabla v + \nabla p = 0, &\quad \textrm{for\;\;} (x,t)\in \mathbb{R}^N\times\mathbb{R},\\
  \mathrm{div}\, v =0, &\quad \textrm{for\;\;} (x,t)\in \mathbb{R}^N\times\mathbb{R},\\
  v|_{t=0}=v_0, & \quad \textrm{for\;\;} x\in \mathbb{R}^N,
\end{cases}
\end{equation}
where $v=(v_1,\cdots,v_N)$ is the vector-valued velocity field and $p$ is the scalar-valued pressure function. The Euler equations \eqref{eq:euler}
describe the motion of the perfect incompressible inviscid fluids and is the fundamental system in the fluid mechanics.

For the smooth data, e.g. $v_0\in H^k(\mathbb{R}^N)$, $k>N/2+2$, it is well-known that there exists a $T>0$ such that
$v\in C(]-T,T[,H^k(\mathbb{R}^N))\cap C^1(]-T,T[; H^{k-1}(\mathbb{R}^N))$
and the pressure satisfies that $-\Delta p= \mathrm{div}\,\mathrm{div} (v\otimes v)$. Up to a function depending only on $t$, the pressure can be given by
\begin{equation}\label{eq:pexp}
  p(x,t)= -\frac{1}{N} |v(x,t)|^2 + \mathrm{p.v.} \int_{\mathbb{R}^N} K_{ij}(x-y) v_i(y,t)v_j(y,t)\,\mathrm{d}y,
\end{equation}
where
\begin{equation}
K_{ij}(y)= \frac{1}{N|\mathbb{S}^{N-1}|}\frac{Ny_i y_j- |y|^2 \delta_{ij}}{|y|^{N+2}}, \quad\textrm{for}\;\; i,j=1,2,\cdots,N
\end{equation}
is the Calder\'on-Zygmund kernel. So far it remains to be an outstanding open problem whether or not we can extend $T$ above to $\infty$ for the smooth solutions of Euler equations.

We here specially focus on the finite-time singularity of self-similar type for the Euler equations.
Such type of singularity is related to the basic property that the equations \eqref{eq:euler} are invariant under the scaling transformation
\begin{equation}
\begin{split}
  & v(x,t)\mapsto v_{\lambda,\alpha}(x,t)=\lambda^\alpha v(\lambda x, \lambda^{1+\alpha} t),\quad \lambda>0,\; \\
  & p(x,t)\mapsto p_{\lambda,\alpha}(x,t)=\lambda^{2\alpha} p(\lambda x, \lambda^{1+\alpha} t).
\end{split}
\end{equation}
In practice, we also combine the spacetime translation in \eqref{eq:euler} to show the exact formula. We call a solution $(v,p)$ of \eqref{eq:euler}
is \textit{(backward) self-similar} with respect to the origin $0$ and time $T$ on the spacetime domain $D:=\mathbb{R}^N\times ]-\infty,T[$
if there exist some $\alpha>-1$ and $T>0$ such that for all $(x,t)\in D$,
\begin{equation}\label{eq:SS}
\begin{split}
  v(x,t)=\lambda(t)^\alpha V\big( \lambda(t)x \big),\quad
  p(x,t)=\lambda(t)^{2\alpha} P\big(\lambda(t) x \big),
\end{split}
\end{equation}
where $\lambda(t)=(T-t)^{-\frac{1}{1+\alpha}}>0$, $(V,P)$ are stationary functions. The assumption $\alpha>-1$ guarantees that the singular solution concentrates on the origin as $t\rightarrow T$.
Up to a spacetime translation, \eqref{eq:SS} corresponds to that for some $\alpha>-1$,
\begin{equation}\label{eq:scal1}
  v(x,t)=v_{\lambda,\alpha}(x,t),\;\;\;p(x,t)= p_{\lambda,\alpha}(x,t),\quad \forall \,\lambda>0, (x,t)\in D.
\end{equation}
A more general case is that the equality \eqref{eq:scal1} holds only for one single $\lambda>1$, and correspondingly we call a solution $(v,p)$ of \eqref{eq:euler}
is \textit{discretely self-similar with a factor $\lambda>1$} with respect to the origin $0$ and time $T$
on the spacetime domain $D:=\mathbb{R}^N\times ]-\infty,T[$ if there exist some $\alpha>-1$ and $T>0$ such that for all $(x,t)\in D$,
\begin{equation}\label{eq:condDSS}
  \mathcal{T}v(x,t)=\mathcal{T}v_{\lambda,\alpha}(x,t),\quad \textrm{for}\;\lambda>1,
\end{equation}
that is,
\begin{equation}
  v(x,T-t)= \lambda^\alpha v(\lambda x, T-\lambda^{1+\alpha} t),\quad \textrm{for}\;\lambda>1,
\end{equation}
where $\mathcal{T}$ is the temporal translation $\mathcal{T} v(x,t)=v(x,T-t)$.
In terms of the similarity variables
\begin{equation}
  y:= \frac{x}{(T-t)^{\frac{1}{1+\alpha}}},\quad s:= \log\Big(\frac{T}{T-t} \Big), \quad \alpha>-1,
\end{equation}
the discretely self-similar solution $(v,p)$ is given by that for all $(x,t)\in \mathbb{R}^N\times ]-\infty,T[$,
\begin{equation}\label{eq:DSS1}
  v(x,t)= \frac{1}{(T-t)^{\frac{\alpha}{1+\alpha}}} V(y,s),
\end{equation}
and
\begin{equation}\label{eq:DSS1p}
  p(x,t)= \frac{1}{(T-t)^{\frac{2\alpha}{1+\alpha}}} P(y,s)+ c(t),
\end{equation}
where $V(y,s)$ and $P(y,s)$ are periodic-in-$s$ functions with the period $$S_0:=(1+\alpha)\log\lambda>0,$$
and $c(t)$ is a function depending only on $t$.
Inserting \eqref{eq:DSS1} into \eqref{eq:euler}, we formally obtain
\begin{equation}\label{eq:DSSEul}
\begin{cases}
  \partial_s V + \frac{\alpha}{\alpha+1} V + \frac{1}{\alpha +1} y\cdot\nabla V + V\cdot\nabla V + \nabla P =0, \\
  \mathrm{div}\, V=0, \\
  V|_{s=0}(y)= T^{\frac{\alpha}{1+\alpha}}v_0(T^{\frac{1}{1+\alpha}} y).
\end{cases}
\end{equation}
Under the mild assumption on $V$, e.g. $V\in L^3_s L^p_y([0,S_0]\times\mathbb{R}^{N})$, $p\in [3,\infty[$ in Theorem \ref{thm:DSS}, from \eqref{eq:pexp} we have
\begin{equation}\label{eq:Pys1}
  P(y,s)=-\frac{1}{N}|V(y,s)|^2 + \mathrm{p.v.} \int_{\mathbb{R}^N} K_{ij}(y-z) V_i(z,s) V_j(z,s)\,\mathrm{d}z.
\end{equation}

Self-similar type singularity plays an important role in the study of singularities,
and has been experimentally detected and theoretically studied in many kinds of partial differential equations (one can refer to the recent survey paper \cite{EggF}).
We here mainly focus on the discretely self-similar singular solution for the Euler equations \eqref{eq:euler}.
Discretely self-similar singularity was firstly introduced by \cite{Chop} in the context of cosmology,
and has been proposed for singularities of the Euler equations (cf. \cite{PomS,PSS}) and other various PDEs (cf. \cite{EggF}).
By definition, discretely self-similar solution \eqref{eq:DSS1}
is a natural generalization of the self-similar solution \eqref{eq:SS}: if the time periodic functions $(V,P)(y,s)$ do not depend on the $s$-variable, i.e., $(V,P)$ are stationary,
it just reduces to the usual self-similar case.

The possibility of the formation of the self-similar singular solutions in the Euler equations \eqref{eq:euler}
and their properties have been intensely studied in the mathematical literature such as
\cite{BroS,Chae1,Chae2,Chae-MA,ChaS,He00,Scho,Shv,Tak,Xue-SS}. But the theoretic study of discretely self-similar solutions for \eqref{eq:euler}
are relatively limited and there are only several recent works on this topic.
Chae and Tsai in \cite{ChaTs} proved some nonexistence results for the discretely self-similar solutions with time-periodic function
$V\in C^1_s C^2_y(\mathbb{R}^{3+1})$ based on the vorticity profile $\Omega=\nabla\times V$: if additionally $|V|$ and $|\nabla V|$ has the decaying asymptotics,
and
\begin{equation}\label{eq:OmCond}
  \Omega\in L^q(\mathbb{R}^3\times [0,S_0]) \quad \textrm{for some  } q\in ]0,3/(1+\alpha)[,
\end{equation}
then $V\equiv 0$ on $\mathbb{R}^{3+1}$.
They also proved the nonexistence results for the time-periodic functions $(V,P)\in C^1_{\textrm{loc}}(\mathbb{R}^{3+1})$
(with $P$ given by \eqref{eq:Pys1}) based on the velocity profile: if
\begin{equation}
\begin{split}
  & V\in L^3(0,S_0;L^r(\mathbb{R}^3)),\;\; r\in [3,9/2],\,\;\;\alpha>3/2, \quad \textrm{or}\\
  & V\in L^2(0,S_0;L^2(\mathbb{R}^3))\cap L^3(0,S_0; L^r(\mathbb{R}^3)), \,\;\; r\in[3,9/2],\;\; -1<\alpha<3/2, \quad \textrm{or}\\
  & V\in L^p(\mathbb{R}^3\times [0,S_0]),\,\;\; p\in [3,\infty[,\,\;\;-1<\alpha\leq 3/p,\quad \textrm{or} \\
  & V\in L^p(\mathbb{R}^3\times [0,S_0]),\;\; p\in [3,\infty[,\;\;\,3/2<\alpha<\infty,
\end{split}
\end{equation}
then $V\equiv 0$ on $\mathbb{R}^{3+1}$. In \cite{Chae3}, by applying the maximum principle in the far field region for the vorticity equations,
Chae proved the following result for the discretely self-similar solutions with the time-periodic vector field
$V\in C^1_s C^2_y(\mathbb{R}^{3+1})$: if additionally
$\sup_{s\in[0,S_0]}|\nabla V(y,s)|=o(1)$ as $|y|\rightarrow \infty$, and there exists $k>\alpha+1$ such that the vorticity profile $\Omega=\nabla\times V$ satisfies
\begin{equation}\label{eq:OmDec}
  |\Omega(y,s)|=O(|y|^{-k}),\quad \textrm{as}\;\, |y|\rightarrow \infty,\;\;\forall s\in [0,S_0],
\end{equation}
then $V(y,s)\equiv C(s)$ for all $y\in\mathbb{R}^3$, where $C:[0,S_0]\rightarrow \mathbb{R}^3$ is
a closed curve satisfying $C(s)=C(s+S_0)$ for all $s\in [0,S_0]$.
Chae in \cite{Chae4} also showed the unique continuation type theorem for the discretely self-similar solutions of \eqref{eq:euler} in $\mathbb{R}^3$.

In this paper we consider the discretely self-similar solutions of the Euler equations \eqref{eq:euler}
to prove some nonexistence results and show the energy behavior of the possible velocity profiles.
The first main result reads as follows, which partially improves the corresponding result of \cite{ChaTs}.
\begin{theorem}\label{thm:DSS}
  Suppose that $V\in C^1_s C^3_{y,\textrm{loc}}(\mathbb{R}^N\times\mathbb{R})$ is a periodic-in-$s$ vector field with period $S_0$,
and $P$ is defined from $V$ by \eqref{eq:Pys1} up to a function depending only on $s$. We have the following statements.
\begin{enumerate}[(1)]
\item
If additionally $V\in L^3 ([0,S_0];L^p(\mathbb{R}^N))$ with some $p\in [3,\infty[$,
then for $\alpha>\frac{N}{2}$ and $-1< \alpha\leq \frac{N}{p}$, we have $V\equiv 0$,
while for $\frac{N}{p} < \alpha \leq \frac{N}{2}$, we have
\begin{equation}\label{eq:conc}
   \sup_{s\in [0,S_0]}\int_{|y|\leq L}|V(y,s)|^2 \mathrm{d}y \lesssim L^{N-2\alpha},\quad \forall L\gg1.
\end{equation}
In particular, for $\frac{N}{p} < \alpha <\frac{N}{2}$, we have either $V\equiv 0$ or
\begin{equation}\label{eq:conc2}
   \int_0^{S_0}\int_{|y|\leq L}|V(y,s)|^2 \mathrm{d}y\mathrm{d}s \sim L^{N-2\alpha},\quad \forall L\gg1.
\end{equation}
\item
For $\alpha= \frac{N}{2}$, if $V\in L^2_{s,y}(\mathbb{R}^N\times [0,S_0])$ (which is slightly weaker than \eqref{eq:conc})
and there exists some constant $0<\delta<1$ such that
\begin{equation}\label{eq:cond3}
  \sup_{s\in[0,S_0]}|V(y,s)|\lesssim |y|^\delta, \quad \forall |y|\gg1,
\end{equation}
then we have
\begin{equation}\label{eq:conc3}
  \int_0^{S_0}\int_{L\leq |y|\leq \lambda L} |V(y,s)|^2 \,\mathrm{d}y\mathrm{d}s \lesssim
  \frac{1}{L^{N+2-\epsilon}},\qquad \forall L\gg 1, \,0<\epsilon\ll 1.
\end{equation}
\end{enumerate}
\end{theorem}

Next we consider the velocity profiles with nondecreasing spatial asymptotics, e.g.,
\begin{equation}\label{eq:asymp1}
  1\lesssim \sup_{s\in [0,S_0]}|V(y,s)|\lesssim |y|^\delta, \quad \forall |y|\gg1,\quad\textrm{for some }\delta\in ]0,1[,
\end{equation}
which are also reasonable and possible candidates: indeed, from the energy equality
$\|v(t)\|_{L^\infty_T L^2_x}= \|v_0\|_{L^2}$ and using the scenario \eqref{eq:DSS1},
we heuristically get
\begin{equation}\label{eq:keyest1}
  \int_{|x|\leq 1}|v(x,t)|^2\,\mathrm{d}x = L^{2\alpha-N} \int_{|y|\leq L}|V(y,s)|^2\,\mathrm{d}y\leq C,\quad \textrm{with}\;\;L=(T-t)^{-\frac{1}{1+\alpha}},
\end{equation}
which corresponds to \eqref{eq:conc} for all $\alpha>-1$,
and thus implies that $V$ possibly can have nondecreasing asymptotics for $-1<\alpha\leq 0$. In order to do so,
we need a refined version of representation formula of the pressure profile in this situation,
since the formula \eqref{eq:Pys1} does not work for the case \eqref{eq:asymp1}. It turns out that the needed representation formula, which is justified in the next section,
can be expressed as (up to a function depending only on $s$)
\begin{equation}\label{eq:Pys0}
  P(y,s)= -\frac{1}{N} |V(y,s)|^2 + p.v.\int_{\mathbb{R}^N} K_{ij}(y-z) V_i(z,s) V_j(z,s)\,\mathrm{d}z + \bar{P}(y,s) + A(s)\cdot y
\end{equation}
where $A(s)\in C(\mathbb{R};\mathbb{R}^N)$ is a fixed periodic-in-$s$ vector-valued function with the period $S_0$ (especially, $A(s)\equiv 0$, if $\alpha>-\frac{1}{2}$ and $\delta<\frac{1}{2}$ in \eqref{eq:asymp1}) and
\begin{equation*}
\bar{P}(y,s)=\left\{
\begin{aligned}
    &-\int_{|z|\geq M}K_{ij}(z) V_i(z,s) V_j(z,s)\,\mathrm{d}z,  & \textrm{if}\;\;1\lesssim \sup_{s\in[0,S_0]}|V(z,s)| \lesssim |z|^\delta,\delta\in [0,\frac{1}{2}[,\, \\
    &-\int_{|z|\geq M} \big(K_{ij}(z)+ y\cdot\nabla K_{ij}(z)\big)V_i V_j(z,s)\,\mathrm{d}z , &  \textrm{if}\;\; |z|^{\frac{1}{2}}\lesssim \sup_{s\in [0,S_0]}|V(z,s)|\lesssim |z|^\delta,\delta\in[\frac{1}{2},1[,
\end{aligned}
\right.
\end{equation*}
with $M>0$ a large number so that \eqref{eq:asymp1} holds for all $|y|\geq M$.
In practice, by using the decompositions like \eqref{eq:decom1}, \eqref{eq:I2Ldec1}, \eqref{eq:I2Ldec2}, it can be proved that $P(y,s)$ defined by \eqref{eq:Pys0} is meaningful and belongs to
$C^0_s C^2_{y,\textrm{loc}}(\mathbb{R}^{N+1})$ under the assumptions \eqref{eq:asymp1} and $V\in C^1_s C^3_{y,\textrm{loc}}(\mathbb{R}^{N+1})$.

Our second main result is as follows.
\begin{theorem}\label{thm:DSS2}
  Suppose that $V\in C^1_s C^3_{y,\mathrm{loc}}(\mathbb{R}^{N+1})$ is a periodic-in-$s$ vector field with period $S_0$,
and $P$ is defined from $V$ through \eqref{eq:Pys0} up to a function depending only on $s$.
\begin{enumerate}[(1)]
\item
If additionally there is a small number $0<\epsilon_0\ll 1$ and some $\delta\in [\epsilon_0,1[$ so that
\begin{equation}\label{eq:Vcond}
  |y|^{\epsilon_0}\lesssim\sup_{s\in [0,S_0]} |V(y,s)| \lesssim |y|^\delta,\qquad \forall |y|\gg 1,
\end{equation}
then the only possible range of $\alpha$ to admit nontrivial velocity profiles is $-\delta\leq \alpha \leq -\epsilon_0$, and the nontrivial profiles corresponding to each $\alpha$ satisfy that
\begin{equation}\label{eq:Vconc}
  \sup_{s\in[0,S_0]}\int_{|y|\leq L}|V(y,s)|^2\,\mathrm{d}y \sim L^{N-2\alpha},\qquad \forall L\gg1.
\end{equation}
\item
If additionally $\alpha>-\frac{1}{2}$ and there is some number $\delta\in]0, \frac{1}{2}[$ so that
\begin{equation}\label{eq:Vcond2}
  1\lesssim \sup_{s\in [0,S_0]} |V(y,s)| \lesssim |y|^\delta,\qquad \forall |y|\gg 1,
\end{equation}
then the only possible range of $\alpha$ to admit nontrivial velocity profiles is $-\delta\leq \alpha \leq 0$,
and the nontrivial profiles corresponding to each $\alpha$ satisfy \eqref{eq:Vconc}.
\end{enumerate}
\end{theorem}

The proofs of Theorem \ref{thm:DSS} and \ref{thm:DSS2} are both based on the local energy inequalities of the velocity profiles \eqref{eq:locEne1}-\eqref{eq:locEne2},
which in turn is derived from the energy equality of the original equality \eqref{eq:EneE}. Then by virtue of a careful treating of the terms containing the pressure profile
(cf. Lemma \ref{lem:pres1} and \ref{lem:pres2}), the proofs are finished through using the bootstrapping method according to the values of $\alpha$ and the assumptions of the velocity profiles.

\begin{remark}\label{rmk1}
  From \eqref{eq:conc2} and \eqref{eq:Vconc}, we can expect that for every $-1<\alpha<\frac{N}{2}$
the corresponding ``typical" possible velocity profiles have the following asymptotics:
\begin{equation*}
  \sup_{s\in [0,S_0]} |V(y,s)|\sim \frac{1}{|y|^\alpha} + o\Big(\frac{1}{|y|^\alpha}\Big),\quad \forall |y|\gg1,
\end{equation*}
and by scaling, we can also expect that
\begin{equation}\label{eq:Omeg}
  \sup_{s\in [0,S_0]} |\Omega(y,s)| \sim \frac{1}{|y|^{\alpha+1}} + o\Big(\frac{1}{|y|^{\alpha+1}}\Big),\quad \forall |y|\gg1,
\end{equation}
with $\Omega:=\nabla\times V$. Note that by comparing \eqref{eq:Omeg} with \eqref{eq:OmCond} and \eqref{eq:OmDec},
we see \eqref{eq:Omeg} is compatible with the nonexistence results of \cite{Chae3,ChaTs} based on the vorticity profiles.
\end{remark}

\begin{remark}\label{rmk2}
  If \eqref{eq:condDSS} holds on the spacetime domain $B_r(0)\times ]-\infty,T[$ with some $r>0$,
then the corresponding solution $(v,p)$ is called the locally discretely self-similar solution.
For such singular solutions, so far it is not clear to show the analogous results as Theorem \ref{thm:DSS} and \ref{thm:DSS2}.
Part of the reason is that the profiles $(V,P)$ are no longer genuinely time periodic functions for $(y,s)\in \mathbb{R}^{N+1}$.
\end{remark}

The outline of this paper is as follows. In Section \ref{sec:pres}, we state and justify the representation formula of the pressure profile in terms of velocity profiles in the considered cases.
In Section \ref{sec:locEne}, we prove the key local energy inequality of the velocity profiles.
Relied on these results, we give the detailed proofs of Theorem \ref{thm:DSS} and \ref{thm:DSS2} in the sections \ref{Sec:thm1} and \ref{sec:Thm2}
respectively. At last we present in Section \ref{sec:lem} two auxiliary and useful lemmas about the terms including the pressure profile.

Throughout this paper, $C$ denotes a harmless constant which may be of different value from line to line. For two quantities $X,Y$,
$X \lesssim Y$ denotes that there is a constant $C>0$ such that $X\leq C Y$, and $X\sim Y$ means that $X\lesssim Y$ and
$Y\lesssim X$. For a real number $a$, denote by $[a]$ its integer part. For $x_0\in \mathbb{R}^N$, $r>0$, denote by $B_r(x_0)$ the open ball of
$\mathbb{R}^N$ centered at $x_0$ with radius $r$, and denote by $B^c_r(x_0)$ its complement set $\mathbb{R}^N\setminus B_r(x_0)$.

\section{Justification of the representation formula of pressure profile}\label{sec:pres}

In this section we justify the needed representation formula of the pressure formula stated at above.
\begin{lemma}\label{lem:VP}
  Suppose $\alpha>-1$, $v$ is a discretely self-similar solution to the Euler equations given by \eqref{eq:DSS1} and the profile
$V\in C^1_s C^3_{y,\mathrm{loc}}(\mathbb{R}^{N+1})$ is a periodic-in-$s$ vector field with period $S_0$.
Then the corresponding pressure profile $P$, which is also periodic-in-$s$ with period $S_0$ and belongs to $ C^0_s C^2_{y,\textrm{loc}}(\mathbb{R}^{N+1})$,
is expressed as (up to a function depending only on $s$)
\begin{equation}\label{eq:Pys}
  P(y,s)= -\frac{1}{N} |V(y,s)|^2 + p.v.\int_{\mathbb{R}^N} K_{ij}(y-z) V_i(z,s) V_j(z,s)\,\mathrm{d}z + \bar{P}(y,s) + A(s)\cdot y,
\end{equation}
 where $\bar{P}(y,s)$ is given by
\begin{equation}\label{eq:Pbar}
\left\{
\begin{aligned}
    &-\int_{|z|\geq M}K_{ij}(z) V_i(z,s) V_j(z,s)\,\mathrm{d}z,  & \textrm{if}\;\;1\lesssim \sup_{s\in[0,S_0]}|V(z,s)| \lesssim |z|^\delta,\delta\in [0,\frac{1}{2}[,\, \\
    &-\int_{|z|\geq M} \big(K_{ij}(z)+ y\cdot\nabla K_{ij}(z)\big)V_i V_j(z,s)\,\mathrm{d}z , &  \textrm{if}\;\; |z|^{\frac{1}{2}}\lesssim \sup_{s\in [0,S_0]}|V(z,s)|\lesssim |z|^\delta,\delta\in[\frac{1}{2},1[,
\end{aligned}
\right.
\end{equation}
and $A(s)\in C(\mathbb{R};\mathbb{R}^N)$ is a fixed vector-valued periodic-in-$s$ function with period $S_0$ satisfying
\begin{equation*}
  A(s)\equiv 0,\quad \textrm{if}\;\;
    1\lesssim \sup_{s\in[0,S_0]}|V(z,s)|\lesssim |z|^\delta, \delta\in [0,\frac{1}{2}[,\;\textrm{and}\;\; \alpha>-\frac{1}{2}.
\end{equation*}
In the above, $M>0$ is a large number so that $1\lesssim \sup_{s\in [0,S_0]}|V(z,s)|\lesssim |z|^\delta$ holds for all $|z|\geq M$.
\end{lemma}

\begin{remark}
  If $V\in L^r_s L^p_y ([0,S_0]\times \mathbb{R}^N)$, $r\in [2,\infty]$, $p\in ]2,\infty[$, then the integral $\textrm{p.v.}\int_{\mathbb{R}^N}K_{ij}(y-z)V_i(z,s)V_j(z,s)\,\mathrm{d}z$
is a meaningful periodic-in-$s$ function belonging to $C^0_s C^2_{y,\textrm{loc}}(\mathbb{R}^{N+1})$, and by using a similar deduction as below,
we can justify that the formula of the pressure profile takes the form \eqref{eq:Pys1} in this case.
\end{remark}

\begin{proof}[Proof of Lemma \ref{lem:VP}]
We here mainly adopt the strategy used in the proof of \cite[Lemma 2.1]{BroS} or \cite[Lemma 2.1]{Xue-SS} with suitable modification.
We first introduce a function $I(y,s)$, which is a part of \eqref{eq:Pys}, and prove that it is meaningfully defined, is a tempered distribution, and it point-wisely solves the Laplace equation
$\Delta I = -\mathrm{div}\mathrm{div}(V\otimes V)$. Then we find a tempered distribution $P(y,s)$ solving the first equation of \eqref{eq:DSSEul}. Since $P$ also solves the same Laplace equation,
the difference between $I$ and $P$ is a harmonic polynomial in the $y$-variable, and at last we prove the order of the polynomial is at most one and show the desired formula \eqref{eq:Pys}.

First define a periodic-in-$s$ function as
\begin{equation}\label{eq:Iy}
  I(y,s)= -\frac{1}{N} |V(y,s)|^2 + p.v.\int_{\mathbb{R}^N} K_{ij}(y-z) V_i(z,s) V_j(z,s)\,\mathrm{d}z + \bar{P}(y,s),\quad
\end{equation}
and we show that $I(y,s)$ is meaningful and is a tempered distribution. Let $\phi_0\in \mathcal{D}(\mathbb{R}^N)$ be a cutoff function supported on
$B_1(0)$ such that $\phi_0\equiv 1$ on $B_{1/2}(0)$ and $0\leq \phi_0\leq 1$. For any $L\geq M$, set
$\phi_L(z)=\phi_0(z/L)$, then we have
\begin{equation}\label{eq:decom1}
  I(y,s)=-\frac{1}{N}|V(y,s)|^2 + I_{1,L}(y,s) + I_{2,L}(y,s),
\end{equation}
with
\begin{equation}
\begin{split}
  & I_{1,L}(y,s) = \textrm{p.v.} \int_{\mathbb{R}^N} K_{ij}(y-z) \phi_{4L}(z) V_i(z,s) V_j(z,s)\,\mathrm{d}z,\quad \textrm{and} \\
  & I_{2,L}(y,s) = \int_{\mathbb{R}^N} K_{ij}(y-z) \big(1-\phi_{4L}(z)\big) V_i(z,s) V_j(z,s)\,\mathrm{d}z + \bar{P}(y,s).
\end{split}
\end{equation}
Since $V\in C^1_s C_{\mathrm{loc}}^3([0,S_0]\times\mathbb{R}^N)$, from the bounded property of the Calder\'on-Zygmund operator,
we infer that $I_{1,L}(y,s)\in C^1_s C^\beta_y$ for all $\beta<3$ with
\begin{equation*}
  \|I_{1,L}\|_{C^1_s C^\beta_y} \lesssim \|V\|_{C^1_s C^3_{\textrm{loc}}}^2.
\end{equation*}
We next consider $I_{2,L}(y,s)$ acting on the ball $B_L(0)$:
if $1\lesssim \sup_{s\in [0,S_0]}|V(z,s)| \lesssim |z|^\delta$, $\delta\in ]0,\frac{1}{2}[$ for all $ |z|\geq M$, from the decomposition
\begin{equation}\label{eq:I2Ldec1}
\begin{split}
  I_{2,L}(y,s) = & \int_{|z|\geq 4L}\big(K_{ij}(y-z)-K_{ij}(z)\big)V_iV_j(z,s)\,\mathrm{d}z - \int_{M\leq|z|\leq 4L}K_{ij}(z)V_iV_j(z,s)\,\mathrm{d}z \\
  & + \int_{2L\leq |z|\leq 4L}K_{ij}(y-z)\big(1-\phi_{4L}(z)\big)V_i(z,s)V_j(z,s)\,\mathrm{d}z,
\end{split}
\end{equation}
then
\begin{equation}\label{eq:I2key}
\begin{split}
  |I_{2,L}(y,s)|\leq & \,\Big|\int_{|z|\geq 4L}\big(K_{ij}(y-z)-K_{ij}(z)\big) V_i V_j(z,s)\mathrm{d}z\Big| +  C\int_{M\leq |z|\leq 4L} \frac{1}{|z|^N} |V(z,s)|^2\,\mathrm{d}z \\
  \lesssim & \int_{|z|\geq 2L} \frac{|y|}{|z|^{N+1}} |V(z,s)|^2\,\mathrm{d}z+ \int_{M\leq |z|\leq 4 L} |z|^{-N+2\delta} \mathrm{d}z\lesssim L^{2\delta};
\end{split}
\end{equation}
and if $|z|^{1/2}\lesssim \sup_{s\in [0,S_0]}|V(z,s)| \lesssim |z|^\delta$, $\delta\in [\frac{1}{2},1[$ for all $ |z|\geq M$, then from the decomposition
\begin{equation}\label{eq:I2Ldec2}
\begin{split}
  I_{2,L}(y,s) = & \int_{|z|\geq 4L}\big(K_{ij}(y-z)-K_{ij}(z)-y\cdot\nabla K_{ij}(z)\big)V_i(z,s)V_j(z,s)\,\mathrm{d}z \,- \\
  & - \int_{M\leq|z|\leq 4L}\big(K_{ij}(z)+y\cdot\nabla K_{ij}(z)\big) V_i(z,s)V_j(z,s)\,\mathrm{d}z \, +  \\
  & + \int_{2L\leq |z|\leq 4L} K_{ij}(y-z)\big(1-\phi_{4L}(z)\big)V_i(z,s)V_j(z,s)\,\mathrm{d}z,
\end{split}
\end{equation}
then
\begin{equation}\label{eq:I2key2}
\begin{split}
  |I_{2,L}(y,s)| \leq & \Big|\int_{|z|\geq 4L} \big(K_{ij}(y-z)-K_{ij}(z)-y\cdot\nabla K_{ij}(z)\big) V_i(z,s) V_j(z,s)\mathrm{d}z\Big| \\
  & \,+\, C\int_{M\leq |z|\leq 4 L} \Big( \frac{1}{|z|^N} + \frac{|y|}{|z|^{N+1}}\Big) |V(z,s)|^2 \mathrm{d}z \\
  \lesssim & \int_{|z|\geq 2L} \frac{|y|^2}{|z|^{N+2}} |V(z,s)|^2\,\mathrm{d}z+ \int_{|z|\sim L} \Big(\frac{1}{|z|^{N-2\delta}}+ \frac{|y|}{|z|^{N+1-2\delta}}\Big) \mathrm{d}z\lesssim L^{2\delta}.
\end{split}
\end{equation}
For $m=1,2$ and for all $y\in B_L(0)$, we also get that
if $1\lesssim \sup_{s\in[0,S_0]}|V(z,s)| \lesssim |z|^\delta$, $\delta\in ]0,\frac{1}{2}[$, $\forall |z|\geq M$, from the decomposition \eqref{eq:I2Ldec1},
\begin{equation*}
\begin{split}
  |\partial^m_y \big(I_{2,L}(y,s)\big)|  &\leq \Big|\partial_y^m\bigg(\int_{|z|\geq 4L} \int_0^1 y\cdot \nabla K_{ij}(\tau y-z) V_iV_j(z,s)\,\mathrm{d}\tau\mathrm{d}z  \bigg)\Big|
  + \int_{|z|\sim L} \frac{C }{|z|^{N+m}} |V(z,s)|^2\,\mathrm{d}z \\
  &\lesssim \int_{|z|\geq 2L} \frac{|y|}{|z|^{N+1+m}} |V(z,s)|^2\,\mathrm{d}z + \int_{|z|\sim L} \frac{1 }{|z|^{N+m}} |V(z,s)|^2\,\mathrm{d}z
  \lesssim L^{-m+2\delta};
\end{split}
\end{equation*}
and if $|z|^{1/2}\lesssim \sup_{s\in [0,S_0]}|V(z,s)| \lesssim |z|^\delta$, $\delta\in [\frac{1}{2},1[$, $\forall |z|\geq M$, from \eqref{eq:I2Ldec2},
\begin{equation*}
\begin{split}
  |\partial^m\big(I_{2,L}(y,s)\big)| \leq &\, \Big|\partial^m_y\bigg(\int_{|z|\geq 4L} \int_0^1 \int_0^1 \Big(y\cdot \nabla^2 K_{ij}(\tau\theta y -z)\cdot y\Big)
  V_i(z,s) V_j(z,s)\tau\mathrm{d}\theta\mathrm{d}\tau\mathrm{d}z\bigg)\Big|  \\
  & +  \Big|\partial_y^m \Big(\int_{M\leq |z|\leq 4L} y\cdot\nabla K_{ij}(z) V_iV_j(z,s) \,\mathrm{d}z\Big)\Big|  + C\int_{|z|\sim L} \frac{1}{|z|^{N+m}} |V(z,s)|^2 \mathrm{d}z \\
  \lesssim & \int_{|z|\geq 2L} \frac{|y|^2}{|z|^{N+2+m}} |V(z,s)|^2\,\mathrm{d}z+ \int_{|z|\sim L} \frac{1}{|z|^{N+m-2\delta}}  \mathrm{d}z
    \\ & \, +
  \begin{cases}
    \int_{M\leq |y|\leq 4L} \frac{1}{|z|^{N+1}} |V(z,s)|^2\,\mathrm{d}z,\quad & \textrm{if}\;\; m=1, \\
    0, \quad & \textrm{if}\;\; m=2,
  \end{cases} \\
 \lesssim  & \;L^{-m+2\delta}.
\end{split}
\end{equation*}
Hence the scalar function $I(y,s)$ defined by \eqref{eq:Iy} is $C^2_y$-smooth on $B_L(0)$ for almost everywhere $s\in [0,S_0]$.
Since $V\in C^1_sC^3_{y,\textrm{loc}}(\mathbb{R}^{N+1})$, and from the above estimates, we can also prove that for every $y\in B_L(0)$, the functions $\partial_y^m I(y,s)$, $m=0,1,2$ are continuous in $s\in [0,S_0]$,
that is, $I(y,s)\in C^0_s C^2_{y,\textrm{loc}}(\mathbb{R}^{N+1})$.
Moreover, for all $y\in B_L(0)$ and a.e. $s\in [0,S_0]$, we have
\begin{equation*}
\begin{split}
  \Delta I & \,= \Delta \Big(-\frac{1}{N} |V|^2\phi_{4L}+I_{1,L}\Big) + \Delta \big(I_{2,L}\big)\\
  & \,= -\mathrm{div}\, \mathrm{div}\big(V\sqrt{\phi_{4L}}\otimes V\sqrt{\phi_{4L}}\big)
  = -\mathrm{div}\, \mathrm{div}\big(V\otimes V\big),
\end{split}
\end{equation*}
where in the second line $\Delta( I_{2,L})=0$ due to that the term $K_{ij}(y-z)-K_{ij}(z)-y\cdot\nabla K_{ij}(z)$ is harmonic in the $y$-variable for all
$y\in B_L(0)$ and $z\in B_{2L}^c(0)$.
Besides, it is not hard to show that $I(y,s)$ is a tempered distribution on $\mathbb{R}^N\times [0,S_0]$: indeed,
we get that
for some $\tilde{p}>2$,
\begin{equation*}
  \int_0^{S_0}\int_{|y|\leq L} |I_{1,L}(y,s)|^{\frac{\tilde{p}}{2}}\,\mathrm{d}y \mathrm{d}s \lesssim
  \int_0^{S_0}\int_{|z|\leq 4L} |V(z,s)|^{\tilde{p}}\,\mathrm{d}z \mathrm{d}s \lesssim L^{N+\tilde{p}\delta},
\end{equation*}
and by \eqref{eq:I2key}, \eqref{eq:I2key2},
\begin{equation*}
  \int_0^{S_0}\int_{|y|\leq L} |I_{2,L}(y,s)|^{\frac{\tilde{p}}{2}}\,\mathrm{d}y\mathrm{d}s  \lesssim L^{N+ \tilde{p}\delta}.
\end{equation*}

Next we intend to find a tempered distributional pressure profile $P(y,s)$ solving the first equation of \eqref{eq:DSSEul}, i.e.,
\begin{equation}\label{eq:keyeq}
  \partial_s V + \frac{\alpha}{1+\alpha} V + \frac{1}{1+\alpha} y\cdot \nabla V + V\cdot\nabla V  + \nabla P = 0.
\end{equation}
Inserting the ansatz \eqref{eq:DSS1} to Euler equations \eqref{eq:euler}, and by setting
\begin{equation}
  y:=\frac{x}{(T-t)^{\frac{1}{1+\alpha}}}, \quad s:=\log\frac{T}{T-t},\quad p(x,t):= \bar p(y,s),
\end{equation}
we obtain that for all $y\in \mathbb{R}^N$, $s\in \mathbb{R}$,
\begin{equation}\label{eq:ueq2}
  \partial_s V(y,s) + \frac{\alpha}{1+\alpha} V(y,s) + \frac{1}{1+\alpha} y\cdot\nabla_y V(y,s) +
  V\cdot\nabla_y V(y,s)  + \nabla_y \Big((T-t)^{\frac{2\alpha}{1+\alpha}} \bar p(y,s)\Big)=0.
\end{equation}
For some $t$ fixed, denoting $f(y,s)= (T-t)^{\frac{2\alpha}{1+\alpha}} \bar p(y,s)=(Te^{-s})^{\frac{2\alpha}{1+\alpha}} \bar p(y,s)$,
then $f(y,s)$ and the vector-valued function $\nabla_y f(y,s)=:g(y,s)$ are periodic-in-$s$
functions with the period $S_0$. Thus from the fundamental theorem of calculus,
we deduce that
\begin{equation}\label{eq:Pysexp}
\begin{split}
  &(T-t)^{\frac{2\alpha}{1+\alpha}} \bar p(y,s) - (T-t)^{\frac{2\alpha}{1+\alpha}} \bar p(0,s)= \\
  =\, & \,f(y,s) -f(0,s)= \int_0^1 \frac{d}{d\tau} f(\tau y,s)\,\mathrm{d}\tau = \int_0^1 y\cdot \nabla f(\tau y,s)\mathrm{d}\tau \\
  =\, & \, \int_0^1 y\cdot g(\tau y,s) \,\mathrm{d}\tau =: P(y,s),
\end{split}
\end{equation}
that is,
\begin{equation}\label{eq:pres2}
  p(x,t)= \frac{1}{(T-t)^{\frac{2\alpha}{1+\alpha}}} P\big(y,s\big) + c(t),\quad \forall \,(x,t)\in \mathbb{R}^N\times ]-\infty,T[.
\end{equation}
with $c(t)= p(0,t)$ and $P(y,s)$ a periodic-in-$s$ function with period $S_0$.
Inserting \eqref{eq:pres2} into \eqref{eq:ueq2} yields the equation \eqref{eq:keyeq} on $\mathbb{R}^N\times \mathbb{R}$.
From the formula of $P(y,s)$ \eqref{eq:Pysexp}, we have $P(y,s)\in C^0_s C^2_{\textrm{loc}}(\mathbb{R}^{N+1})$.
Next we prove that $P(y,s)$ is a tempered distribution of $[0,S_0]\times\mathbb{R}^N$.
The proof is similar to that in \cite[Lemma 2.1]{BroS}, but we here sketch it for completeness.
Since we have the energy conservation of the original velocity and \eqref{eq:pexp},
we infer that $\|p(x,t)\|_{L^1_{\textrm{weak}}}\lesssim \|v(x,t)\|_{L^2}^2 \lesssim \|v_0\|_{L^2}^2 \lesssim 1$, which means that
$|\{x: |p(x,t)|>\lambda\}| \leq \frac{C}{\lambda}$ for all $t<T$. Thus there exists a small number $\eta >0$ so that
$|\{x: |p(x,t)|> \frac{1}{\eta}\frac{1}{ (T-t)^{N/(1+\alpha)}}\}|\leq \frac{|B_1(0)|}{2}(T-t)^{\frac{N}{1+\alpha}}$, which yields that there is a point $x_t$ in the ball
$\{x: |x|\leq (T-t)^{\frac{1}{1+\alpha}}\}$ so that $|p(x_t,t)|\leq \frac{1}{\eta}\frac{1}{ (T-t)^{N/(1+\alpha)}}$. Hence with $x_t$ and the corresponding $y_t=\frac{x_t}{(T-t)^{1/(1+\alpha)}}\in B_1(0)$ at our disposal,
we have
$$|c(t)|\leq (T-t)^{-\frac{2\alpha}{1+\alpha}} |P(y_t,s)| + \eta^{-1}(T-t)^{-\frac{N}{1+\alpha}} \lesssim (T-t)^{-\frac{2\alpha}{1+\alpha}} + (T-t)^{-\frac{N}{1+\alpha}}, $$
where we have used the fact that $|P(y_t,s)|\leq C$ from $P(y,s)\in C^0_s C^2_{\textrm{loc}}(\mathbb{R}^{N+1})$.
From \eqref{eq:pres2}, we see that
\begin{equation*}
  P(y,s)=(T-t)^{\frac{2\alpha}{1+\alpha}} p\big(y (T-t)^{\frac{1}{1+\alpha}},t\big) + (T-t)^{\frac{2\alpha}{1+\alpha}} c(t), \quad \forall (y,s)\in \mathbb{R}^N\times [0,S_0],
\end{equation*}
thus if $1\lesssim \sup_{s\in [0,S_0]}|V(y,s)|\lesssim |y|^\delta$, $\forall |y|\geq M$ for some $\delta\in [0,1[$, we get that for some $\tilde{p}\in]2,\infty[$,
\begin{equation}\label{eq:estim1}
\begin{split}
  & \int_{\frac{1}{4(T-t)^{1/(1+\alpha)}}\leq |y|\leq \frac{1}{2(T-t)^{1/(1+\alpha)}}}|P(y,s)|^{\frac{\tilde{p}}{2}}\,\mathrm{d}y\,\\
  \lesssim &\, 1+  (T-t)^{ \frac{(2\alpha-N)\tilde{p}/2}{1+\alpha}} + (T-t)^{\frac{\tilde{p}\alpha-N}{1+\alpha}}
  \int_{\frac{1}{4}\leq |x|\leq \frac{1}{2}}|p(x,t)|^{\frac{\tilde{p}}{2}}\,\mathrm{d}x\, \\
  \lesssim &\, 1+ (T-t)^{\frac{(2\alpha-N)\tilde{p}/2}{1+\alpha}} + (T-t)^{\frac{\tilde{p}\alpha-N}{1+\alpha}}
  \bigg(\int_{\frac{1}{8}\leq |x|\leq 1}|v(x,t)|^{\tilde{p}}\,\mathrm{d}x + \|v(t)\|_{L^2}^{\tilde{p}}\bigg) \\
  \lesssim & \,1+ (T-t)^{\frac{(2\alpha-N)\tilde{p}/2}{1+\alpha}} + (T-t)^{\frac{\tilde{p}\alpha-N}{1+\alpha}} +
  (T-t)^{\frac{\tilde{p}\alpha}{1+\alpha}}\int_{\frac{1}{8(T-t)^{1/(1+\alpha)}}\leq |y|\leq \frac{1}{(T-t)^{1/(1+\alpha)}}}|V(y,s)|^{\tilde{p}}\,\mathrm{d}y,
\end{split}
\end{equation}
which leads to that
\begin{equation}\label{eq:estim2}
  \sup_{s\in [0,S_0]}\bigg(\int_{\frac{1}{4(T-t)^{1/(1+\alpha)}}\leq |y|\leq \frac{1}{2(T-t)^{1/(1+\alpha)}}}|P(y,s)|^{\frac{\tilde{p}}{2}}\,\mathrm{d}y\bigg)\lesssim (T-t)^{-m_1}.
\end{equation}
In the above deduction of \eqref{eq:estim1} from the second line to the third line, we have used the decomposition that for $\frac{1}{4}\leq |x|\leq \frac{1}{2}$,
\begin{equation*}
\begin{split}
  p(x,t) & = -\frac{1}{N}|v(x,t)|^2 + \Big(\int_{|z|\leq \frac{1}{8}} + \int_{\frac{1}{8}\leq |z|\leq 1} +
  \int_{|z|\geq 1} \Big) \Big(K_{ij}(x-z)v_i(z,t)v_j(z,t)\,\mathrm{d}z\Big) \\
  & := -\frac{1}{N} |v(x,t)|^2 + p_1(x,t) + p_2(x,t) + p_3(x,t),
\end{split}
\end{equation*}
and the following estimates that
\begin{equation*}
\begin{split}
  \|p_1(x,t)\|_{L^\infty_x(\{\frac{1}{4}\leq |x|\leq \frac{1}{2} \})}\lesssim \|v(x,t)\|_{L^2_x}^2,  \\
  \int_{\frac{1}{4}\leq |x|\leq \frac{1}{2}}|p_2(x,t)|^{\frac{\tilde{p}}{2}}\,\mathrm{d}x \lesssim \int_{\frac{1}{8}\leq |x|\leq 1}|v(x,t)|^{\tilde{p}}\,\mathrm{d}x, \\
  \|p_3(x,t)\|_{L^\infty_x(\{\frac{1}{4}\leq |x|\leq \frac{1}{2} \})}\lesssim \int_{|z|\geq 1} \frac{1}{|z|^N} |v(z,t)|^2\,\mathrm{d}z \lesssim \|v(x,t)\|_{L^2_x}^2.
\end{split}
\end{equation*}
According to \eqref{eq:estim2}, we infer that $P(y,s)$ is a tempered distribution of $\mathbb{R}^{N+1}$.

Now we show that $P(y,s)$ and $I(y,s)$ are equal up to a first-order harmonic polynomial about the $y$-variable.
Since they both satisfy the Laplace equation $\Delta I =-\mathrm{div}\mathrm{div}(V\otimes V)= \Delta P$,
and are both tempered distributions on $\mathbb{R}^N\times \mathbb{R}$, the difference
\begin{equation}
  P(y,s) - I(y,s)=:h(y,s)
\end{equation}
is a harmonic polynomial about the $y$-variable, e.g. $h(y,s)$ may take the form
$$a(s)+b_i(s)y_i + c_{ij}(s) y_iy_j+d_{ijk}(s) y_i y_j y_k+\cdots,\qquad i,j,k=1,\cdots,N,$$
with the coefficients depending only on $s$.
Since $P(y,s)$ and $I(y,s)$ both belonging to $C^0_s C^2_{y,\textrm{loc}}(\mathbb{R}^{N+1})$ are periodic-in-$s$ with period $S_0$,
$h(y,s)$ is also a periodic function belonging to this functional space such that $h(y,s+S_0)=h(y,s)$ for all $s\in\mathbb{R}$.
In the following we prove that the order of $h(y,s)$ is at most one,
and in some case $h(y,s)$ is a function depending only on the $s$-variable.
For all $|y|\leq \frac{1}{2(T-t)^{\frac{1}{1+\alpha}}}$, from \eqref{eq:DSS1}, \eqref{eq:pexp}, and the change of variables, we see that
\begin{equation*}
\begin{split}
  &(T-t)^{\frac{2\alpha}{1+\alpha}} p\big(y (T-t)^{\frac{1}{1+\alpha}} , t\big)  = \\
  =\,& -\frac{1}{N} |V(y,s)|^2
  + \textrm{p.v.} \int_{|z|\leq 1} K_{ij}\big(y(T-t)^{\frac{1}{1+\alpha}}-z\big)\;(V_i V_j)
  \Big(\frac{z}{(T-t)^{\frac{1}{1+\alpha}}},s\Big)\mathrm{d}z \\
  & \,+ (T-t)^{\frac{2\alpha}{1+\alpha}} \int_{|z|\geq 1} K_{ij}\big(y(T-t)^{\frac{1}{1+\alpha}}-z\big)\;
  (v_i v_j)(z,t)\,\mathrm{d}z + d_1(t) \\
  = \, & -\frac{1}{N} |V(y,s)|^2
  + \textrm{p.v.} \int_{|z|\leq  (T-t)^{-\frac{1}{1+\alpha}}} K_{ij}(y-z)(V_i V_j)\big(z,s\big)\mathrm{d}z + \tilde{p}(y,t) + d_1(t), \\
  = \, &\,  I(y,s) -\tilde{I}(y,s) + \tilde{p}(y,t) + d_1(t),
\end{split}
\end{equation*}
with
\begin{equation*}
  \tilde{p}(y,t):= (T-t)^{\frac{2\alpha}{1+\alpha}} \int_{|z|\geq 1} K_{ij}\big(y(T-t)^{\frac{1}{1+\alpha}}-z\big)
  (v_i v_j)(z,t)\,\mathrm{d}z,
\end{equation*}
and
\begin{equation*}
  \tilde{I}(y,s,t) := \int_{|z|\geq (T-t)^{-\frac{1}{1+\alpha}}} K_{ij}(y-z) V_i(z,s) V_j(z,s)\,\mathrm{d}z + \bar P\big(y,s\big).
\end{equation*}
On the other hand, from \eqref{eq:pres2}, we also have
\begin{equation}\label{eq:fact10}
  (T-t)^{\frac{2\alpha}{1+\alpha}} p\big(y (T-t)^{\frac{1}{1+\alpha}}, t\big)= P(y,s) + d_2(t),
\end{equation}
with $d_2(t):=(T-t)^{\frac{2\alpha}{1+\alpha}} c(t)$;
hence we deduce
\begin{equation}
  h(y,s) + d_2(t)-d_1(t) = -\tilde{I}(y,s,t) + \tilde{p}(y,t), \quad \forall |y|\leq \frac{1}{2}(T-t)^{-\frac{1}{1+\alpha}},\forall s\in[0,S_0],
\end{equation}
which implies that
\begin{equation}\label{eq:key2}
  \Big|h(y,s)+ d_2(t) - d_1(t)\Big| \leq |\tilde{p}(y,t)| + \sup_{s\in [0,S_0]}\big|\tilde{I}(y,s,t)\big|, \qquad \forall |y|\leq \frac{1}{2}(T-t)^{-\frac{1}{1+\alpha}},\,\forall s\in [0,S_0].
\end{equation}
For $\tilde{p}$, from $|y(T-t)^{\frac{1}{1+\alpha}}-z|\geq |z|- |y(T-t)^{\frac{1}{1+\alpha}}|\geq 1/2$ for all $z\in B^c_1(0)$,
and the energy conservation of $v$, we directly obtain
\begin{equation}
  |\tilde{p}(y,t)|\lesssim (T-t)^{\frac{2\alpha}{1+\alpha}} \|v(t)\|_{L^2}^2\lesssim (T-t)^{\frac{2\alpha}{1+\alpha}}.
\end{equation}
For $\tilde{I}$,
if $1\lesssim \sup_{s\in [0,S_0]}|V(y,s)|\lesssim |y|^\delta$, $\forall |y|\geq M$ for some $\delta\in[0,1[$, similarly as the treating of \eqref{eq:I2key} and \eqref{eq:I2key2}, we get
\begin{equation}
  \sup_{s\in[0,S_0]}|\tilde{I}(y,s,t)| \lesssim (T-t)^{-\frac{2\delta}{1+\alpha}}, \quad  \forall |y|\leq \frac{1}{2}(T-t)^{-\frac{1}{1+\alpha}}.
\end{equation}
Since $\alpha>-1$, $\delta\in[0,1[$ and \eqref{eq:key2} holds for all $|y|\leq \frac{1}{2}(T-t)^{-\frac{1}{1+\alpha}}$ and $s\in [0,S_0]$,
we infer that the order of harmonic polynomial $h(y,s)$ is at most one.
In particular, if $\alpha>-1/2$, $\delta\in[0,1/2[$, we have that $h(y,s)$ is a function depending only on the $s$-variable.
\end{proof}

\section{Local energy inequality of the velocity profiles}\label{sec:locEne}

We start with the local energy equality of the original velocity
\begin{equation}\label{eq:EneE}
\begin{split}
  &\,\int_{\mathbb{R}^N} |v(x,t_2)|^2 \chi(x,t_2)\,\mathrm{d}x - \int_{\mathbb{R}^N} |v(x,t_1)|^2 \chi(x,t_1)\,\mathrm{d}x \\
  =&\,\int_{t_1}^{t_2}\int_{\mathbb{R}^N} |v(x,t)|^2 \partial_t \chi(x,t)\,\mathrm{d}x \mathrm{d}t\,
  +\,\int_{t_1}^{t_2}\int_{\mathbb{R}^N}\Big(|v|^2v + 2 \big(p-c(t)\big) v\Big)\cdot\nabla \chi(x,t)\,\mathrm{d}x \mathrm{d}t,
\end{split}
\end{equation}
with $-\infty<t_1<t_2<T$ and $\chi\in \mathcal{D}(\mathbb{R}^N\times ]-\infty,T[)$. The equality can hold if the velocity field is regular enough,
e.g. $v\in C^1_{\mathrm{loc}}(\mathbb{R}^N\times ]-\infty,T[)\cap L^\infty(]-\infty,T[; L^2(\mathbb{R}^N))$.

Let $\phi\in \mathcal{D}(\mathbb{R}^N)$ be a cutoff function supported on $B_1(0)$ such that $\phi\equiv 1$ on $B_{1/\lambda}(0)$
and $0\leq \phi\leq 1$ ($\lambda>1$ is just the DSS factor in \eqref{eq:condDSS}). Set $\chi(x,t)=\phi(x)$, then for any $t_1<t_2<T$, \eqref{eq:EneE} reduces to
\begin{equation}\label{eq:EneE2}
\begin{split}
  &\,\int_{\mathbb{R}^N} |v(x,t_2)|^2 \phi(x)\,\mathrm{d}x - \int_{\mathbb{R}^N} |v(x,t_1)|^2 \phi(x)\,\mathrm{d}x \\
  =&\,\int_{t_1}^{t_2}\int_{\mathbb{R}^N}\Big(|v|^2v + 2\big(p-c(t)\big)v\Big)(x,t)\cdot\nabla \phi(x) \,\mathrm{d}x \mathrm{d}t.
\end{split}
\end{equation}
Inserting the ansatz \eqref{eq:DSS1} into \eqref{eq:EneE2}, and denoting $s_2:=\log\frac{1}{T-t_2}$, $s_1:=\log\frac{1}{T-t_1}$,
we obtain that for any $-\infty<s_1<s_2<\infty$,
\begin{equation}\label{eq:EneE3}
\begin{split}
  & e^{s_2\frac{2\alpha-N}{1+\alpha}} \int_{\mathbb{R}^N}|V(y,s_2)|^2 \phi\big(y e^{-\frac{1}{1+\alpha}s_2}\big)\,\mathrm{d}y
  - e^{s_1 \frac{2\alpha-N}{1+\alpha}} \int_{\mathbb{R}^N} |V(y,s_1)|^2 \phi\big(y e^{-\frac{1}{1+\alpha}s_1}\big)\,\mathrm{d}y \\
  = \,& \int_{t_1}^{t_2}\int_{\mathbb{R}^N} \frac{1}{(T-t)^{\frac{3\alpha-N}{1+\alpha}}}
  \big(|V|^2V+2PV(y,s) \big)\cdot\nabla \phi\big(y(T-t)^{\frac{1}{1+\alpha}}\big)\,\mathrm{d}y\mathrm{d}t \\
  =\, & \int_{s_1}^{s_2}\int_{\mathbb{R}^N} e^{s \frac{2\alpha-N-1}{1+\alpha}}
  \big(|V|^2V+2PV(y,s)\big)\cdot\nabla\phi\big(ye^{-s\frac{1}{1+\alpha}}\big)\,\mathrm{d}y\mathrm{d}s.
\end{split}
\end{equation}
With no loss of generality, we assume that $s_1+S_0<s_2$ with $S_0$ the period. Let $\tau_1,\tau_2\in [0,S_0]$ be arbitrary,
and by replacing $s_i$ with $s_i+\tau_i$ in \eqref{eq:EneE3} (if $s_1+S_0\geq s_2$, we may use $s_1+\tau_1$ and $s_1+\tau_1+\tau_2$
to replace $s_1$ and $s_2$ respectively), we get
\begin{equation}\label{eq:EneE4}
\begin{split}
  & e^{(s_2+\tau_2)\frac{2\alpha-N}{1+\alpha}}\int_{\mathbb{R}^N}|V(y,s_2+\tau_2)|^2 \phi\big(y e^{-\frac{s_2+\tau_2}{1+\alpha}}\big)\,\mathrm{d}y
  -e^{(s_1+\tau_1)\frac{2\alpha-N}{1+\alpha}} \int_{\mathbb{R}^N} |V(y,s_1+\tau_1)|^2\phi\big(y e^{-\frac{s_1+\tau_1}{1+\alpha}}\big)\,\mathrm{d}y \\
  =\,& \int_{s_1+\tau_1}^{s_2+\tau_2}\int_{\mathbb{R}^N} e^{\frac{2\alpha-N-1}{1+\alpha}} \big(|V|^2V +2P V(y,s)\big)\cdot\nabla
  \phi\big(y e^{-s\frac{1}{1+\alpha}}\big) \,\mathrm{d}y\mathrm{d}s.
\end{split}
\end{equation}
For $i=1,2$, we set
\begin{equation}\label{eq:IJ}
\begin{split}
  &\, I_i := \int_0^{S_0} \int_{\mathbb{R}^N} e^{(s_i+\tau_i)\frac{2\alpha-N}{1+\alpha}} |V(y,s_i+\tau_i)|^2
  \phi\big(y e^{-\frac{s_i+\tau_i}{1+\alpha}}\big)\,\mathrm{d}y \mathrm{d}\tau_i, \\
  &\, J_i := \sup_{\tau_i\in [0,S_0]} \int_{\mathbb{R}^N} e^{(s_i+\tau_i)\frac{2\alpha-N}{1+\alpha}} |V(y,s_i+\tau_i)|^2
  \phi\big(y e^{-\frac{s_i +\tau_i }{1+\alpha}}\big)\,\mathrm{d}y .
\end{split}
\end{equation}
By the periodicity property of $V$, and denoting
\begin{equation}
  l_i:= e^{\frac{s_i}{1+\alpha}}\in ]0,\infty[, \quad \mu:= e^{\frac{S_0}{1+\alpha}},
\end{equation}
we directly see that for $i=1,2$,
\begin{equation}\label{eq:IJest}
\begin{split}
  c_\alpha l_i^{2\alpha-N} \int_0^{S_0} \int_{|y|\leq \frac{l_i}{\lambda}} |V(y,s)|^2\,\mathrm{d}y\mathrm{d}s \leq\,
  & I_i\leq C_\alpha l_i^{2\alpha-N}\int_0^{S_0} \int_{|y|\leq \mu l_i} |V(y,s)|^2 \,\mathrm{d}y\mathrm{d}s, \\
  c_\alpha l_i^{2\alpha-N} \sup_{s\in[0,S_0]} \int_{|y|\leq \frac{l_i}{\lambda}}|V(y,s)|^2\,\mathrm{d}y \leq\,
  & J_i \leq C_\alpha l_i^{2\alpha-N} \,\sup_{s\in [0,S_0]} \int_{|y|\leq \mu l_i} |V(y,s)|^2\,\mathrm{d}y,
\end{split}
\end{equation}
where $c_\alpha:=\min \{1,e^{S_0\frac{2\alpha-N}{1+\alpha}}\}$, $C_\alpha:=\max\{1,e^{S_0\frac{2\alpha-N}{1+\alpha}}\}$.
Taking the supremum over the $\tau_2$-variable and integrating on the $\tau_1$-variable in \eqref{eq:EneE4}, we obtain
\begin{equation*}
  |S_0 J_2 -I_1| \,\leq K_1
\end{equation*}
with
\begin{equation*}
  K_1\,:= \sup_{\tau_2\in[0,S_0]} \int_0^{S_0} \int_{s_1+\tau_1}^{s_2+\tau_2}\int_{ \mathbb{R}^N}e^{s\frac{2\alpha-N-1}{1+\alpha}}
  \big(|V|^3 + 2 |PV|(y,s)\big)
  \big|\nabla \phi\big( y e^{-\frac{s}{1+\alpha}}\big)\big|\,\mathrm{d}y\mathrm{d}s\mathrm{d}\tau_1.
\end{equation*}
Denoting by
\begin{equation}
  k_1:= [\log_\lambda (\mu l_2/l_1)],\quad B_k:=\{s:l_1 \lambda^k\leq e^{\frac{s }{1+\alpha}}\leq l_1\lambda^{k+1} \},
\end{equation}
and from the support property of $\phi$ and the periodicity property of $(V,P)$,
%$$|\{s : |y|\leq e^{\frac{s}{1+\alpha}}\leq 2|y| \}|=\{s: (1+\alpha )\log |y|\leq s\leq (1+\alpha)(\log 2 +\log |y|)\}\leq (1+\alpha) \log2$$
we infer that
\begin{equation}\label{eq:K1K2}
\begin{split}
  K_1 & \,\leq S_0 \int_{s_1}^{s_2+S_0} \int_{\mathbb{R}^N}
  e^{s\frac{2\alpha-N-1}{1+\alpha}}
  \big(|V|^3 +2 |PV|\big)(y,s)\, \big|\nabla\phi(y e^{-\frac{s}{1+\alpha}})\big| \,\mathrm{d}y\mathrm{d}s \\
  & \, \leq S_0 \sum_{k=0}^{k_1} \int_{s_1}^{s_2}\int_{\frac{1}{\lambda}e^{\frac{s}{1+\alpha}}\leq |y|\leq e^{\frac{s}{1+\alpha}}} 1_{B_k}(s) \, e^{s\frac{2\alpha-N-1}{1+\alpha}}
  \big(|V|^3 + 2 |PV|\big)(y,s)\, \big|\nabla \phi(y e^{-\frac{s}{1+\alpha}})\big|\,\mathrm{d}y\mathrm{d}s \\
  & \, \leq S_0 \sum_{k=0}^{k_1} \frac{1}{(l_1\lambda^k)^{N+1-2\alpha}}\int_{(1+\alpha)\log (l_1\lambda^k)}^{(1+\alpha)\log(l_1\lambda^{k+1})}
  \int_{l_1\lambda^{k-1}\leq |y|\leq l_1 \lambda^{k+1}} \frac{|V|^3 + 2|P| |V|}{|y|^{N+1-2\alpha}}
  \big|\nabla\phi(y e^{-\frac{s}{1+\alpha}})\big|\,\mathrm{d}y\mathrm{d}s \\
  & \, \leq \frac{C S_0 }{\lambda} \sum_{k=0}^{k_1}\frac{1}{(l_1 \lambda^k)^{N+1-2\alpha}}\int_0^{S_0} \int_{l_1\lambda^{k-1}\leq |y|\leq l_1\lambda^{k+1}} \big(|V|^3 + |P| |V|\big)(y,s) \,\mathrm{d}y\mathrm{d}s = : K_2,
\end{split}
\end{equation}
where in the last line we have used the fact that $|B_k|= (1+\alpha)\log\lambda =S_0$. On the other hand, we also get
\begin{equation*}
\begin{split}
  K_1 & \leq C S_0\int_{s_1}^{s_2+S_0}\int_{\frac{1}{\lambda}e^{\frac{s}{1+\alpha}}\leq |y|\leq e^{\frac{s}{1+\alpha}}}
  \frac{|V|^3 + |P| |V(y,s)|}{|y|}|\nabla\phi(y e^{-\frac{s}{1+\alpha}})|\,\mathrm{d}y\mathrm{d}s \\
  & \leq C S_0 \int_{s_1}^{s_2 +S_0} \int_{\frac{1}{\lambda}l_1\leq |y|\leq \mu l_2} 1_{\{s:\frac{1}{\lambda}|y|\leq e^{\frac{s}{1+\alpha}}\leq |y|\}}\frac{|V|^3 + |P| |V(y,s)|}{|y|}|\nabla \phi(y e^{-\frac{s}{1+\alpha}})|\,\mathrm{d}y\mathrm{d}s\\
  & \leq \frac{C S_0}{\lambda}\int_0^{S_0} \int_{\frac{1}{\lambda}l_1\leq |y|\leq \mu l_2} \frac{|V|^3 +|P||V(y,s)|}{|y|}\,\mathrm{d}y\mathrm{d}s
  =: K_3,
\end{split}
\end{equation*}
where the interval $\{s: |y|/\lambda\leq e^{\frac{s}{1+\alpha}}\leq |y|\}$ has the length $(1+\alpha)\log\lambda=S_0 $
and is just of a period. Hence we find
\begin{equation}\label{eq:locEne1}
  |S_0 J_2 -I_1| \leq C K_2,\quad\textrm{and}\quad |S_0J_2-I_1|\leq C K_3.
\end{equation}
Similarly, by using the different treating of $\tau_1,\tau_2$ in \eqref{eq:EneE4}, i.e. taking the supremum norm and $L^1$-norm on the $\tau_1,\tau_2$ variables in different order, we also have
\begin{equation}\label{eq:locEne2}
  |I_2 -I_1| + |I_2-S_0 J_1| + |J_2-J_1| \leq C K_2,
\end{equation}
and
\begin{equation}\label{eq:locEne3}
  |I_2-I_1| + |I_2-S_0 J_1| + |J_2-J_1|\leq C K_3.
\end{equation}

\section{Proof of Theorem \ref{thm:DSS}}\label{Sec:thm1}

\subsection{Proof of Theorem \ref{thm:DSS}-(1)}

 We here mainly focus on the case of $\alpha>\frac{N}{p}$, especially $\frac{N}{p}<\alpha \leq \frac{N}{2}$.
We start with the inequality \eqref{eq:locEne1}: $|S_0 J_2- I_1|\leq C K_2$ (or the inequality $|J_2-J_1|\leq C K_2$),
and by setting $l_1=\lambda$ and $l_2=\lambda L \gg 1$, we get
\begin{equation*}
\begin{split}
  & \, L^{2\alpha-N} \sup_{s\in[0,S_0]}\int_{|y|\leq L} |V(y,s)|^2\,\mathrm{d}y \\
  \lesssim\, &  \int_0^{S_0}\int_{|y|\leq \lambda\mu} |V(y,s)|^2\,\mathrm{d}y\mathrm{d}s
  + \sum_{k=0}^{[\log_\lambda(\mu L)]}\frac{1}{\lambda^{k(N+1-2\alpha)}}\int_0^{S_0} \int_{|y|\sim \lambda^k} \big(|V|^3 + |P| |V|\big)(y,s) \,\mathrm{d}y\mathrm{d}s.
\end{split}
\end{equation*}
It directly leads to
\begin{equation}
  \sup_{s\in[0,S_0]}\int_{|y|\leq L} |V(y,s)|^2\,\mathrm{d}y  \leq C L^{N-2\alpha} + C E(L),
\end{equation}
with
\begin{equation*}
  E(L):= L^{N-2\alpha} \sum_{k=0}^{[\log_\lambda (\mu L)]}\frac{1}{\lambda^{k(N+1-2\alpha)}}\int_0^{S_0} \int_{ |y|\sim \lambda^k} \big(|V|^3 + |P| |V|\big)(y,s)\,\mathrm{d}y\mathrm{d}s.
\end{equation*}
Thanks to H\"older's inequality and Lemma \ref{lem:pres1} below, we see that
\begin{equation}\label{eq:EL}
\begin{split}
  E(L) \lesssim & L^{N-2\alpha}\sum_{k=0}^{[\log_\lambda(\mu L)]} \frac{\lambda^{k(N-3N/p)}}{\lambda^{k(N+1-2\alpha)}}
  \int_0^{S_0}\bigg(\Big(\int_{|y|\sim \lambda^k} |V|^p\,\mathrm{d}y\Big)^{3/p} +
  \Big(\int_{|y|\sim \lambda^k}|P|^{\frac{p}{2}}\mathrm{d}y\Big)^{3/p}\bigg)\mathrm{d}s \\
  \lesssim & L^{N-2\alpha} \sum_{k=0}^{[\log_\lambda (\mu L)]} \frac{1}{\lambda^{k(N+1-2\alpha)}}\lambda^{k(N-3N/p)}
  \int_0^{S_0}\Big(\int_{|y|\lesssim \lambda^k}|V(y,s)|^p\,\mathrm{d}y\Big)^{3/p}\mathrm{d}s \\
  \lesssim & L^{N-2\alpha} \sum_{k=0}^{[\log_\lambda (\mu L)]} \lambda^{-k(1-2\alpha+\frac{3N}{p})}.
\end{split}
\end{equation}
If $1-2\alpha +\frac{3N}{p}\geq 0$, i.e. $2\alpha\leq 1+\frac{3N}{p}$, from the above estimate of $E(L)$ we get
\begin{equation*}
  \sup_{s\in[0,S_0]}\int_{|y|\leq L} |V(y,s)|^2\,\mathrm{d}y \leq C L^{N-2\alpha} [\log_\lambda L].
\end{equation*}
Otherwise, if $1-2\alpha+\frac{3N}{p}<0$, we obtain
\begin{equation}\label{eq:Vest1}
  \sup_{s\in [0,S_0]} \int_{|y|\leq L} |V(y,s)|^2\,\mathrm{d}y \leq C L^{\beta_p}, \qquad \textrm{with}\;\;\beta_p:=N-1-\frac{3N}{p}>0.
\end{equation}
Next we intend to improve the estimate \eqref{eq:Vest1} in this case. By interpolation and H\"older's inequality, we infer that
\begin{equation}\label{eq:V3Inter}
\begin{split}
  \int_0^{S_0}\int_{|y|\leq L} |V(y,s)|^3\,\mathrm{d}y\mathrm{d}s &
  \lesssim \int_0^{S_0}\Big(\int_{|y|\leq L} |V(y,s)|^2\,\mathrm{d}y\Big)^{\theta_p}
  \Big(\int_{|y|\leq L}|V(y,s)|^p\,\mathrm{d}y\Big)^{1-\theta_p} \,\mathrm{d}s \\
  & \lesssim \Big(\sup_{s\in[0,S_0]}\int_0^{S_0} |V(y,s)|^2\,\mathrm{d}y\Big)^{\theta_p} \Big(\int_0^{S_0 }\Big(\int_{|y|\leq L}
  |V(y,s)|^p\,\mathrm{d}y\Big)^{\frac{3}{p}}\,\mathrm{d}s\Big)^{\frac{p}{3(p-2)}} \\
  & \lesssim L^{\beta_p\theta_p},
\end{split}
\end{equation}
with $\theta_p:=\frac{p-3}{p-2}$. We use this estimate and Lemma \ref{lem:pres1} to improve the bound of $E(L)$ as follows
\begin{equation}\label{eq:EL2}
\begin{split}
  E(L) & \lesssim L^{N-2\alpha} \sum_{k=0}^{[\log_\lambda(\mu L)]}\frac{1}{\lambda^{N+1-2\alpha}}\int_0^{S_0}\Big(\int_{|y|\sim \lambda^k}|V|^3\,\mathrm{d}y
  + \int_{|y|\sim \lambda^k}|P|^{\frac{3}{2}}\,\mathrm{d}y\Big)\mathrm{d}s \\
  & \lesssim  L^{N-2\alpha} \sum_{k=0}^{[\log_\lambda (\mu L)]} \frac{1}{\lambda^{N+1-\alpha}} \int_0^{S_0}\int_{|y|\lesssim \lambda^k}
  |V(y,s)|^3\,\mathrm{d}y\mathrm{d}s \\
  & \lesssim L^{N-2\alpha}\sum_{k=0}^{[\log_\lambda (\mu L)]} \lambda^{-k (N+1-2\alpha- \beta_p \theta_p )}.
\end{split}
\end{equation}
If $N+1-2\alpha -\beta_p \theta_p \geq 0$, i.e. $N-2\alpha\geq \beta_p \theta_p-1$, then we directly get
\begin{equation*}
  \sup_{s\in [0,S_0]}\int_{|y|\leq L} |V(y,s)|^2\,\mathrm{d}y \lesssim L^{N-2\alpha} [\log_\lambda L].
\end{equation*}
Otherwise if $N+1-2\alpha -\beta_p <0$, then we obtain that
\begin{equation*}
  \sup_{s\in [0,S_0]} \int_{|y|\leq L} |V(y,s)|^2\,\mathrm{d}y \lesssim L^{\beta_p \theta_p -1},
\end{equation*}
and thus by interpolation,
\begin{equation*}
  \int_0^{S_0}\int_{|y|\leq L} |V(y,s)|^3\,\mathrm{d}y\mathrm{d}s \lesssim L^{\beta_p \theta_p^2 -\theta_p}.
\end{equation*}
The above process can be iteratively repeated in finite time, and for every $\alpha>\frac{N}{p}$,
there exists $n\in \mathbb{N}$ so that $N-2\alpha-(\beta_p\theta_p^n-\theta_p^{n-1} -\cdots-1) \geq 0$, and we have
\begin{equation}\label{eq:V2est1}
  \sup_{s\in[0,S_0]}\int_{|y|\leq L}|V(y,s)|^2\,\mathrm{d}y \lesssim L^{N-2\alpha} [\log_\lambda L].
\end{equation}
By passing $L$ to $\infty$, this already implies that $V\equiv 0$ for all $(y,s)\in \mathbb{R}^{N+1}$ at the case $\alpha>\frac{N}{2}$.

Now we remove the additional term $[\log_\lambda L]$ appearing in \eqref{eq:V2est1}.
Let $\epsilon\in ]0,1[$, then from \eqref{eq:V2est1} we deduce that
\begin{equation*}
  \sup_{s\in [0,S_0]} \int_{|y|\leq L} |V(y,s)|^2\,\mathrm{d}y \lesssim_\epsilon L^{N-2\alpha+\epsilon},
\end{equation*}
and by interpolation,
\begin{equation*}
  \int_0^{S_0} \int_{|y|\leq L} |V(y,s)|^3\,\mathrm{d}y\mathrm{d}s \lesssim L^{(N-2\alpha+\epsilon)\theta_p}.
\end{equation*}
Similarly as obtaining \eqref{eq:EL2} we have
\begin{equation*}
\begin{split}
  E(L) & \lesssim L^{(N-2\alpha)} \sum_{k=0}^{[\log_\lambda (\mu L)]} \lambda^{-k \big(N+ 1-2\alpha -(N-2\alpha+\epsilon)\theta_p\big)} \\
  & \lesssim L^{N-2\alpha} \sum_{k=0}^{[\log_\lambda (\mu L)]} \lambda^{-k(1-\epsilon)}\lesssim L^{N-2\alpha}.
\end{split}
\end{equation*}
Hence the desired estimate \eqref{eq:conc} is derived for any $\alpha\in ]\frac{N}{p},\frac{N}{2}]$.

Next we consider the statement \eqref{eq:conc2} for the case $\frac{N}{p}<\alpha <\frac{N}{2}$. In order to show \eqref{eq:conc2}
for the nontrivial velocity profile, it suffices to prove the following inequality
\begin{equation}\label{eq:V2targ}
  \frac{1}{L^{N-2\alpha}} \int_0^{S_0}\int_{|y|\leq L} |V(y,s)|^2\,\mathrm{d}y\mathrm{d}s \gtrsim 1,\quad \forall L\gg1.
\end{equation}
The method is by contradiction. Suppose \eqref{eq:V2targ} does not hold, then there is a sequence of numbers $L_n\gg 1$ such that as
$L_n\rightarrow \infty$, one has
\begin{equation*}
  \frac{1}{L_n^{N-2\alpha}} \int_0^{S_0}\int_{|y|\leq L_n} |V(y,s)|^2\,\mathrm{d}y\mathrm{d}s \rightarrow 0.
\end{equation*}
We shall use the local energy inequality $|I_2-S_0 J_1|\leq C K_2$, and by setting $l_2= L_n\rightarrow \infty$ and $l_1=\lambda L$, we get
\begin{equation}\label{eq:V2est2}
  \sup_{s\in [0,S_0]} \int_{|y|\leq L}|V(y,s)|^2\,\mathrm{d}y\lesssim L^{N-2\alpha}
  \sum_{k=0}^\infty \frac{1}{(\lambda^k L)^{N+1-2\alpha}}\int_0^{S_0}\int_{|y|\sim \lambda^k L} \big(|V|^3 + |P| |V|\big)(y,s)\,\mathrm{d}y\mathrm{d}s.
\end{equation}
Since we already have \eqref{eq:conc}, thanks to the inequality \eqref{eq:V3Inter}, we deduce
\begin{equation*}
  \int_0^{S_0}\int_{|y|\leq L} |V|^3\,\mathrm{d}y\mathrm{d}s\lesssim L^{(N-2\alpha)\theta_p},\quad \forall L\gg 1.
\end{equation*}
By virtue of Lemma \ref{lem:pres1} again (similar to obtaining \eqref{eq:EL2}), we find that
\begin{equation}\label{eq:V2est3}
  \sup_{s\in [0,S_0]} \int_{|y|\leq L} |V(y,s)|^2\,\mathrm{d}y\,\lesssim \frac{1}{L} \sum_{k=0}^\infty 2^{-k(N-2\alpha+1)} (2^k L)^{(N-2\alpha)\theta_p} \lesssim L^{(N-2\alpha)\theta_p -1}.
\end{equation}
By interpolation we further get
\begin{equation*}
  \int_0^{S_0}\int_{|y|\leq L} |V(y,s)|^3\,\mathrm{d}y\mathrm{d}s \lesssim L^{(N-2\alpha)\theta_p^2-\theta_p}.
\end{equation*}
Using this improved estimate in \eqref{eq:V2est2} we further obtain a more refined estimate than \eqref{eq:V2est3}. By repeating such iterative process,
after a finite $n$-times, we obtain
\begin{equation*}
  \sup_{s\in [0,S_0]}\int_{|y|\leq L} |V(y,s)|^2\,\mathrm{d}y \lesssim L^{(N-2\alpha)\theta_p^n -\theta_p^{n-1}-\cdots -1}.
\end{equation*}
For $n$ large enough, the power of $L$ becomes negative, which guarantees $\sup_{s\in [0,S_0]}\int_{\mathbb{R}^N} |V(y,s)|^2\,\mathrm{d}y\equiv0$,
and thus $V\equiv 0$ for all $(y,s)\in \mathbb{R}^{N+1}$.

For the case $-1<\alpha\leq N/p$, since the treating is similar to the obtaining of \eqref{eq:conc2} or that of the case $-1<\alpha<-\delta$ in the next section,
we omit the details and we only note that for all $-1<\alpha\leq N/p$,
\begin{equation*}
\begin{split}
  \frac{1}{l_2^{N-2\alpha}}\int_0^{S_0} &\int_{|y|\leq \mu l_2}|V|^2\,\mathrm{d}y\mathrm{d}s \leq \frac{1}{l_2^{N-2\alpha}}\int_0^{S_0}\int_{|y|\leq M}|V|^2\,\mathrm{d}y\mathrm{d}s +
  \frac{1}{l_2^{N-2\alpha}}\int_0^{S_0}\int_{M\leq |y|\leq \mu l_2} |V|^2\,\mathrm{d}y\mathrm{d}s \\
  & \lesssim \frac{ M^{N(1-\frac{2}{p})}}{l_2^{N-2\alpha}}\bigg(\int_0^{S_0} \Big(\int_{|y|\leq M}|V|^p \mathrm{d}y\Big)^{\frac{3}{p}}\mathrm{d}s\bigg)^{\frac{2}{3}} + l_2^{2(\alpha-\frac{N}{p})}\bigg(\int_0^{S_0}
  \Big(\int_{|y|\geq M}|V|^p \mathrm{d}y\Big)^{\frac{2}{p}}\mathrm{d}s\bigg)^{\frac{2}{3}} \\
  & \rightarrow 0, \quad\quad \textrm{as}\;\; l_2\rightarrow \infty, \;\; \textrm{and then}\;\; M\rightarrow \infty,
\end{split}
\end{equation*}
and at the first step of iteration
\begin{equation*}
\begin{split}
  \int_0^{S_0}\int_{|y|\sim \lambda^k L} \big(|V|^3+ |P| |V|\big)\,\mathrm{d}y\mathrm{d}s &\lesssim (\lambda^k L)^{N(1-\frac{3}{p})}\int_0^{S_0}\bigg(\Big(\int_{|y|\sim \lambda^k L}|V|^p\,\mathrm{d}y\Big)^{\frac{3}{p}} + \Big(\int_{|y|\sim \lambda^k L} |P|^{\frac{p}{2}}\,\mathrm{d}y\Big)^{\frac{3}{p}}\Big)\bigg) \\
  & \lesssim (\lambda^k L)^{N(1-\frac{2}{p})} \int_0^{S_0}\Big(\int_{|y|\lesssim \lambda^k L} |V|^p\,\mathrm{d}y\Big)^{\frac{3}{p}}\,\mathrm{d}s \lesssim (\lambda^k L)^{N(1-3/p)}.
\end{split}
\end{equation*}

\subsection{Proof of Theorem \ref{thm:DSS}-(2)}

We begin with the local energy inequality \eqref{eq:locEne3}: $|I_2-I_1|\leq C K_3$ at the $\alpha= N/2$ case,
and from \eqref{eq:IJ} and $c_\alpha=C_\alpha=1$ for $\alpha= N/2$, it also leads to
\begin{equation*}
  \int_0^{S_0}\int_{|y|\leq l_2/\lambda} |V|^2\,\mathrm{d}y\mathrm{d}s \leq \int_0^{S_0} \int_{|y|\leq \mu l_1}|V|^2\,\mathrm{d}y\mathrm{d}s + C
  \int_0^{S_0}\int_{ \frac{l_1}{\lambda}\leq |y|\leq \mu l_2} \frac{|V|^3 + |P| |V(y,s)|}{|y|}\,\mathrm{d}y\mathrm{d}s.
\end{equation*}
By letting $\mu l_1=L$ and $l_2=\lambda^2 L$, we get
\begin{equation}
\begin{split}
  \int_0^{S_0}\int_{L\leq |y|\leq\lambda L} |V(y,s)|^2\,\mathrm{d}y \mathrm{d}s & \leq \frac{C}{L} \int_0^{S_0}
  \int_{\frac{L}{\mu\lambda} \leq |y|\leq \mu \lambda^2 L} \big(|V|^3 + |P| |V|\big)(y,s)\,\mathrm{d}y\mathrm{d}s \\
  & \leq \frac{C}{L}\int_0^{S_0}\int_{ \frac{L}{\lambda^{\nu+1}} \leq |y|\leq \lambda^{\nu+2} L}
  \big(|V|^3 + |P| |V(y,s)|\big)\,\mathrm{d}y\mathrm{d}s
\end{split}
\end{equation}
with $\nu:=[\log_\lambda \mu]+1$.
From \eqref{eq:cond3}, we see that
\begin{equation*}
  \int_0^{S_0}\int_{L\leq |y|\leq\lambda L}|V|^2\,\mathrm{d}y\mathrm{d}s \leq \frac{C}{L^{1-\delta}}\int_0^{S_0}\int_{\frac{L}{\lambda^{\nu+1}}\leq |y|\leq \lambda^{\nu+2}L}|V|^2\,\mathrm{d}y\mathrm{d}s
  + \frac{C}{L}\int_0^{S_0}\int_{\frac{L}{\lambda^{\nu+1}}\leq |y|\leq \lambda^{\nu+2}L} |P| |V|\,\mathrm{d}y\mathrm{d}s.
\end{equation*}
In order to treat the term involving $P$, we make the following decomposition
\begin{equation*}
\begin{split}
  P(y,s) = & \, -\frac{|V(y,s)|^2}{N} + \mathrm{p.v.}\int_{|z|\leq \frac{L}{\lambda^{\nu+2}}}K_{ij}(y-z)V_i(z,s)V_j(z,s)\,\mathrm{d}z \\
  & + \int_{\frac{L}{\lambda^{\nu+2}}\leq|z|\leq 2^{\nu+3}L} K_{ij}(y-z)V_i(z,s)V_j(z,s)\,\mathrm{d}z
  +\int_{|z|\geq \lambda^{\nu+3}L}K_{ij}(y-z)V_i(z,s)V_j(z,s)\,\mathrm{d}z \\
  :=&\, P_{1,L}(y,s) + P_{2,L}(y,s) + P_{3,L}(y,s) + P_{4,L}(y,s).
\end{split}
\end{equation*}
The treating of the term containing $P_{1,L}$ is obvious:
\begin{equation*}
  \frac{1}{L}\int_0^{S_0}\int_{\frac{L}{\lambda^{\nu+1}}\leq |y|\leq\lambda^{\nu+2}L}|P_{1,L}||V|\,\mathrm{d}y\mathrm{d}s
  \leq \frac{C}{L^{1-\delta}}\int_0^{S_0}\int_{\frac{L}{\lambda^{\nu+1}}\leq |y|\leq \lambda^{\nu+2}L} |V(y,s)|^2\,\mathrm{d}y\mathrm{d}s.
\end{equation*}
For $P_{2,L}$, from the support property,
we infer that for every $|y|\geq \frac{L}{2^{\nu+1}}$,
\begin{equation*}
  |P_{2,L}(y,s)|\leq \frac{C}{L^N} \int_{|z|\leq \frac{L}{\lambda^{\nu+2}}}|V(z,s)|^2\,\mathrm{d}z\mathrm{d}s
  \leq \frac{C \|V\|^2_{L^2_s L^2_y}}{L^N},
\end{equation*}
and thus by the H\"older inequality we obtain
\begin{equation*}
\begin{split}
  \frac{1}{L}\int_0^{S_0}\int_{\frac{L}{\lambda^{\nu+1}}\leq |y|\leq \lambda^{\nu+2}L}|P_{2,L}| |V| \,\mathrm{d}y\mathrm{d}s & \leq \frac{C}{L^{N+1}}
  \int_{\frac{L}{\lambda^{\nu+1}}\leq |y|\leq \lambda^{\nu+2}L} |V|\,\mathrm{d}y\mathrm{d}s \\
  & \leq \frac{C}{L^{N/2+1}} \Big(\int_{\frac{L}{\lambda^{\nu +1}}\leq |y|\leq \lambda^{\nu+2}L}|V|^2\,\mathrm{d}y\mathrm{d}s\Big)^{\frac{1}{2}}.
\end{split}
\end{equation*}
For $P_{3,L}$, taking advantage of the Calder\'on-Zygmund theorem and \eqref{eq:cond3} again, we find that
\begin{equation*}
\begin{split}
  & \frac{1}{L}\int_0^{S_0}\int_{\frac{L}{\lambda^{\nu+1}} \leq |y|\leq \lambda^{\nu+2} L} |P_{3,L}| |V|\,\mathrm{d}y\mathrm{d}s \\
  \leq\, & \frac{1}{L}\Big(\int_0^{S_0}\int_{\frac{L}{\lambda^{\nu+1}}\leq |y|\leq \lambda^{\nu+2}L}|V|^2\,\mathrm{d}y\mathrm{d}s\Big)^{\frac{1}{2}}
  \Big(\int_0^{S_0} \int_{\frac{L}{\lambda^{\nu+1}}\leq |y|\leq \lambda^{\nu+2}L}|P_3|^2\,\mathrm{d}y\mathrm{d}s\Big)^{\frac{1}{2}} \\
  \leq \, & \frac{C}{L}\Big(\int_0^{S_0}\int_{\frac{L}{\lambda^{\nu+1}}\leq |y|\leq \lambda^{\nu+2}L}|V|^2\,\mathrm{d}y\mathrm{d}s\Big)^{\frac{1}{2}}
  \Big(\int_0^{S_0}\int_{\frac{L}{\lambda^{\nu+2}}\leq |y|\leq \lambda^{\nu+3}L}|V|^4\,\mathrm{d}y\mathrm{d}s\Big)^{\frac{1}{2}} \\
  \leq \, & \frac{C}{L^{1-\delta}}\int_0^{S_0}\int_{\frac{L}{\lambda^{\nu+2}}\leq |y|\leq \lambda^{\nu+3}L}|V|^2\,\mathrm{d}y\mathrm{d}s.
\end{split}
\end{equation*}
By virtue of the dyadic decomposition and \eqref{eq:cond3}, we estimate the term containing $P_{4,L}$ as follows
\begin{equation*}
\begin{split}
  & \frac{1}{L} \int_0^{S_0}\int_{\frac{L}{\lambda^{\nu+1}}\leq |y|\leq \lambda^{\nu+2}L}
  |P_{4,L}| |V|\,\mathrm{d}y\mathrm{d}s \\
  \leq\, & L^{N-1+\delta}
  \int_0^{S_0}\sup_{|y|\leq \lambda^{\nu+2}L}|P_{4,L}(y,s)|\,\mathrm{d}s \\
  \leq\, & C L^{N-1+\delta}\int_0^{S_0} \sup_{|y|\leq \lambda^{\nu+2}L} \Big(\sum_{k=\nu+3}^\infty\int_{\lambda^k L\leq |z|\leq \lambda^k L}
  \frac{1}{|y-z|^N}|V(z,s)|^2\,\mathrm{d}z\Big)\mathrm{d}s \\
  \leq\, & \frac{C}{L^{1-\delta}} \sum_{k=\nu+3}^\infty \frac{1}{\lambda^{Nk}} \int_0^{S_0}\int_{\lambda^k L\leq |z|\leq \lambda^{k+1}L}
  |V(z,s)|^2\,\mathrm{d}z\mathrm{d}s.
\end{split}
\end{equation*}
Gathering the above estimates leads to
\begin{equation}\label{eq:V2key}
\begin{split}
  \int_0^{S_0}\int_{L\leq |y|\leq 2L} |V|^2\,\mathrm{d}y\mathrm{d}s \leq & \frac{C}{L^{N/2+1}}\sum_{j=-\nu-1}^{\nu+2}\Big(\int_0^{S_0}
  \int_{\lambda^j L\leq |y|\leq \lambda^{j+1}L}|V|^2\,\mathrm{d}y\mathrm{d}s\Big)^{\frac{1}{2}} \\
  & + \frac{C}{L^{1-\delta}}\sum_{k=-\nu-2}^\infty \frac{1}{\lambda^{N k }}\int_0^{S_0}
  \int_{\lambda^k L\leq |y|\leq \lambda^{k+1}L} |V|^2 \,\mathrm{d}y\mathrm{d}s.
\end{split}
\end{equation}
By denoting $A_k= A_k(L):= \int_0^{S_0}\int_{\lambda^k L\leq |y|\leq \lambda^{k+1} L} |V|^2\,\mathrm{d}y\mathrm{d}s$ for every $k\in\mathbb{Z}$, we rewrite \eqref{eq:V2key} as
\begin{equation}\label{eq:A0}
  A_0\leq \frac{C}{L^{N/2+1}} \sum_{j=-\nu-1}^{\nu+2} A_j^{1/2} + \frac{C}{L^{1-\delta}} \sum_{k=-\nu-2}^\infty \frac{1}{\lambda^{Nk}} A_k,
\end{equation}
which also ensures that for every $i\in\mathbb{Z}$,
\begin{equation}\label{eq:Ai}
  A_i \leq \frac{C}{(\lambda^i L)^{N/2+1}} \sum_{j=-\nu-1}^{\nu+2} A_{i+j}^{1/2} + \frac{C}{L^{1-\delta}} \sum_{k=-\nu-2}^\infty \frac{1}{\lambda^{Nk}}
  A_{i+k}.
\end{equation}
Using \eqref{eq:Ai} in estimating the righthand side of \eqref{eq:A0}, we get
\begin{equation*}
\begin{split}
  A_0 \leq & \frac{C}{L^{N/2+1}}\sum_{j_1=-\nu-1}^{\nu+2} \bigg(\frac{C}{(\lambda^{j_1}L)^{(N/2+1)/2}}\sum_{j_2=-\nu-1}^{\nu+2} A_{j_1+j_2}^{1/4} + \frac{C}{L^{(1-\delta)/2}}
  \sum_{k_2=-\nu-2}^\infty \frac{1}{\lambda^{Nk_2/2}} A_{j_1+k_2}^{1/2}\bigg) \\
  & \,+ \frac{C}{L^{1-\delta}} \sum_{k_1=-\nu-2}^\infty\frac{1}{\lambda^{Nk_1}} \bigg(\frac{C}{(\lambda^{k_1}L)^{N/2+1}}\sum_{j_2=-\nu-1}^{\nu+2}A_{k_1+j_2}^{1/2} + \frac{1}{L^{1-\delta}}\sum_{k_2=-\nu-2}^\infty
  \frac{1}{\lambda^{Nk_2}} A_{k_1+k_2}\bigg) \\
  \leq & \frac{C}{L^{(N/2+1)(1+1/2)}} \sum_{J_1,J_2=-\nu-1}^{\nu+2} A_{j_1+j_2}^{1/4} + \frac{C}{L^{N/2+1+(1-\delta)/2}} \sum_{j_1=-\nu-1}^{\nu+2}\sum_{k_2=-\nu-2}^\infty\frac{1}{\lambda^{Nk_2/2}} A_{j_1+k_2}^{1/2} \\
  & \,+ \frac{C}{L^{N/2+2-\delta}} \sum_{k_1=-\nu-2}^\infty \sum_{j_2=-\nu-1}^{\nu+2} \frac{1}{\lambda^{ k_1 N}} A_{k_1+j_2}^{1/2} + \frac{C}{L^{2(1-\delta)}}\sum_{k_1,k_2=-\nu-2}^\infty \frac{1}{\lambda^{N(k_1+k_2)}}
  A_{k_1+k_2}.
\end{split}
\end{equation*}
By repeating this process for $n$-times, we obtain
\begin{equation*}
\begin{split}
  A_0 \leq & \frac{C}{L^{(N/2+1)(1+\cdots+1/2^n)}}\sum_{j_1,\cdots,j_n=-\nu-1}^{\nu+2} A_{j_1+\cdots+j_n}^{1/2^{n+1}} \\
  & \, + \frac{C}{L^{(N/2+1)(1+\cdots+ 2^{n-1})+ (1-\delta)/2^n}} \sum_{j_1,\cdots,j_{n-1}=-\nu-1}^{\nu+2}\sum_{k_n=-\nu-2}^\infty \frac{1}{\lambda^{Nk_2/2^n}} A_{j_1+\cdots+ j_{n-1}+k_n}^{1/2^n} \\
  & \, + \cdots + \frac{C}{L^{N/2+1+n(1-\delta)}}\sum_{k_1,\cdots,k_{n-1}=-\nu-2}^\infty\sum_{j_n=-\nu-1}^{\nu+2}\frac{1}{\lambda^{(k_1+\cdots+k_{n-1})N}} A_{k_1+\cdots+k_{n-1}+j_n}^{1/2} \\
  & \, + \frac{C}{L^{(n+1)(1-\delta)}} \sum_{k_1,\cdots,k_n=-\nu-2}^\infty \frac{1}{\lambda^{(k_1+\cdots+k_n)N}} A_{k_1+\cdots+k_n}.
\end{split}
\end{equation*}
For every small $\epsilon>0$, due to $A_k \leq C$ for all $k\in \mathbb{Z}$, we can let $n$ large enough so that
\begin{equation*}
  A_0(L)\leq \frac{C}{L^{(N/2+1)(1+\cdots+2^n)}} +\cdots + \frac{C}{L^{(N/2+1)(1+\cdots+1/2^{m-1})+ (n+1-m)(1-\delta)/2^m}}+\cdots+\frac{C}{L^{(n+1)(1-\delta)}} \leq \frac{C}{L^{N+2-\epsilon}}.
\end{equation*}
This concludes the proof of \eqref{eq:conc3}.

\section{Proof of Theorem \ref{thm:DSS2}}\label{sec:Thm2}

\subsection{Proof of Theorem \ref{thm:DSS2}-(1)}

Since by a simple deduction in the introduction section we already have \eqref{eq:keyest1} for all $\alpha>-1$, that is,
\begin{equation}\label{eq:keyest2}
  \sup_{s\in [0,S_0]} \int_{|y|\leq L}|V(y,s)|^2\,\mathrm{d}y \lesssim L^{N-2\delta},\qquad \forall L\gg 1,
\end{equation}
we infer that the only possible scope of $\alpha$ to admit nontrivial velocity profiles is $-1<\alpha\leq -\epsilon_0$,
which can be seen from \eqref{eq:keyest2} and the following fact deduced by the assumption \eqref{eq:Vcond}:
\begin{equation}\label{eq:fact0}
  \sup_{s\in[0,S_0]} \int_{|y|\leq L} |V(y,s)|^2 \,\mathrm{d}y \geq \int_{M\leq |y|\leq L} |y|^{2\epsilon_0}\,\mathrm{d}y \gtrsim L^{N+2\epsilon_0},
\end{equation}
with $M>0$ a large number so that \eqref{eq:Vcond} holds for all $|y|\geq M$. We remark that by starting from $|J_2-J_1|\leq C K_2$ and in a similar way as the treating in
the corresponding part of \cite{Xue-SS}, we can also prove \eqref{eq:keyest2} in the same style as conducted in the main proof (noting that the assumption \eqref{eq:Vcond} is still necessary),
but we here omit the details for simplicity.

Next we consider the case $-1<\alpha<-\delta$, and we begin with the local energy inequality
$|J_2-J_1|\leq C K_2$. Thanks to \eqref{eq:Vcond}, we see that for all $\alpha\in ]-1,-\delta[$,
\begin{equation*}
\begin{split}
  \frac{1}{l_2^{N-2\alpha}}\int_0^{S_0} \int_{|y|\leq \mu l_2} |V(y,s)|^2\,\mathrm{d}y\mathrm{d}s & \lesssim l_2^{-N+2\alpha}
  \int_0^{S_0}\int_{ |y|\leq\mu l_2} |y|^{2\delta}\,\mathrm{d}y\mathrm{d}s\\
  & \lesssim l_2^{2\delta + 2\alpha}\rightarrow 0,\quad\quad\quad \textrm{as}\;\; l_2\rightarrow \infty,
\end{split}
\end{equation*}
thus by letting $l_1=\lambda L\gg1$ and $l_2\rightarrow \infty$ and using \eqref{eq:IJest}, we have
\begin{equation}\label{eq:key1}
  \sup_{s\in [0,S_0]}\int_{|y|\leq L} |V(y,s)|^2\,\mathrm{d}y
  \leq C L^{N-2\alpha} \sum_{k=0}^\infty \frac{1}{(\lambda^k L)^{N+1-2\alpha}}\int_0^{S_0}\int_{\lambda^k L\leq |y|\leq \lambda^{k+2} L}
  \big(|V|^3 + |P| |V|\big)\,\mathrm{d}y\mathrm{d}s,
\end{equation}
where $P$ is given by \eqref{eq:Pys0}. 
Taking advantage of the following rough estimate deduced from \eqref{eq:Vcond},
\begin{equation}\label{eq:Ves-rou}
  \sup_{s\in [0,S_0]}\int_{|y|\leq \lambda^{k+2}L}|V(y,s)|^2\,\mathrm{d}y\lesssim
  (\lambda^k L)^{N+2\delta},\quad \textrm{with}\;\;\delta\in ]0,1[,
\end{equation}
and by using \eqref{eq:estq2} in Lemma \ref{lem:pres2} below, we have
\begin{equation*}
\begin{split}
  \int_0^{S_0}\int_{|y|\leq \lambda^{k+2}L} |V(y,s)| |P(y,s)| \,\mathrm{d}y \mathrm{d}s & \lesssim
  \begin{cases}
    (\lambda^k L)^{N+3\delta} + (\lambda^k L)^{N+\delta +1},\quad &\textrm{if}\;\; \delta\neq\frac{1}{2}, \\
    (\lambda^k L)^{N+\frac{3}{2}}[\log_2 (\lambda^k L)],\quad &\textrm{if}\;\; \delta=\frac{1}{2},
  \end{cases} \\
  & \lesssim
  \begin{cases}
  (\lambda^k L)^{N+3\delta},\quad &\textrm{if   }\delta> \frac{1}{2}, \\
  (\lambda^k L)^{N+\frac{3}{2}+\epsilon}, \quad & \textrm{if   }\delta=\frac{1}{2},\\
  (\lambda^k L)^{N+\delta+1} ,\quad &\textrm{if    } \delta<\frac{1}{2},
  \end{cases}
\end{split}
\end{equation*}
with $0<\epsilon\ll 1/2$ a small number. Thus for all $-1<\alpha<-\delta$, we first obtain a bound which is better than \eqref{eq:Ves-rou}:
\begin{equation}\label{eq:Ves-ref}
\begin{split}
  \sup_{s\in [0,S_0]}\int_{|y|\leq L} |V(y,s)|^2 \mathrm{d}y & \leq
  \begin{cases}
    \frac{C}{L} \sum_{k=0}^\infty \lambda^{-k(N-2\alpha+1)} (\lambda^k L)^{\max\{N+3\delta,N+1+\delta \}},\quad& \textrm{if}\;\; \delta\neq \frac{1}{2}, \\
    \frac{C}{L} \sum_{k=0}^\infty \lambda^{-k(N-2\alpha+1)} (\lambda^k L)^{N+\frac{3}{2}+\epsilon},\quad& \textrm{if}\;\; \delta= \frac{1}{2}, \\
  \end{cases} \\
  & \leq
  \begin{cases}
  C L^{N+3\delta-1},\quad &\textrm{if  }\delta\in ]\frac{1}{2},1[, \\
  C L^{N+\delta+\epsilon} ,\quad &\textrm{if   } \delta\in ]0,\frac{1}{2}].
  \end{cases}
\end{split}
\end{equation}
We next shall use \eqref{eq:Ves-ref} to show a more refined estimate. By using \eqref{eq:estq2} in Lemma \ref{lem:pres2} again, and noting that
\begin{equation}\label{eq:fact1}
  \max\{b+\delta, (N+b)/2+1\}=
  \begin{cases}
    b+\delta,\quad & \textrm{if    }b\geq N+2(1-\delta), \\
    \frac{N+b}{2}+1,\quad & \textrm{if   }b<N+2(1-\delta),
  \end{cases}
\end{equation}
we get
\begin{equation}\label{eq:est4}
  \int_0^{S_0}\int_{|y|\leq \lambda^{k+2}L} |V(y,s)| |P(y,s)|\,\mathrm{d}y \mathrm{d}s \lesssim
  \begin{cases}
  (\lambda^k L)^{N+4\delta-1},\quad &\textrm{if  }\delta\in [\frac{3}{5},1[, \\
  (\lambda^k L)^{N+\frac{3\delta+1}{2}},\quad &\textrm{if  }\delta\in ] \frac{1}{2}, \frac{3}{5}], \\
  (\lambda^k L)^{N+\frac{\delta+\epsilon}{2}+1},\quad &\textrm{if  }\delta\in ]0, \frac{1}{2}],
  \end{cases}
\end{equation}
Plugging it into \eqref{eq:key1}, we have
\begin{equation}\label{eq:G2}
\begin{split}
  \sup_{s\in [0,S_0]}\int_{|y|\leq L} |V(y,s)|^2 \,\mathrm{d}y  & \leq
  \begin{cases}
  \frac{C}{L}\sum_{k=0}^\infty \lambda^{-k(N-2\alpha+1)}  (\lambda^k L)^{N+4\delta-1}, \quad &\textrm{if  }\delta\in [\frac{3}{5},1[,\\
  \frac{C}{L}\sum_{k=0}^\infty \lambda^{-k(N-2\alpha+1)}  (\lambda^k L)^{N+\frac{3\delta+1}{2}},\quad &\textrm{if  }\delta\in ] \frac{1}{2}, \frac{3}{5}], \\
  \frac{C}{L}\sum_{k=0}^\infty \lambda^{-k(N-2\alpha+1)}  (\lambda^k L)^{N+\frac{\delta+\epsilon}{2}+1},\quad &\textrm{if  }\delta\in ]0, \frac{1}{2}],
  \end{cases} \\
  & \leq
  \begin{cases}
   C L^{N+4\delta-2},\quad &\textrm{if  }\delta\in [\frac{3}{5},1[, \\
   C L^{N+\frac{3\delta-1}{2}},\quad &\textrm{if  }\delta\in ] \frac{1}{2}, \frac{3}{5}], \\
   C L^{N+\frac{\delta+\epsilon}{2}},\quad &\textrm{if  }\delta\in ]0, \frac{1}{2}].
  \end{cases}
\end{split}
\end{equation}
We can repeat the above process for $n+1$ times to show that
\begin{equation}\label{eq:keyeq3}
  \sup_{s\in[0,S_0]}\int_{|y|\leq L} |V(y,s)|^2\,\mathrm{d}y \leq
  \begin{cases}
    C L^{N+2\delta+ (n+1)(\delta-1)},\quad & \textrm{if   }\delta\in [\frac{n+2}{n+4},1[, \\
    C L^{N+\frac{2\delta+ n(\delta-1)}{2}},\quad &\textrm{if   }\delta\in [\frac{n+1}{n+3}, \frac{n+2}{n+4}], \\
    C L^{N+\frac{2\delta + (n-1)(\delta-1)}{2^2}},\quad & \textrm{if   }\delta\in [\frac{n}{n+2},\frac{n+1}{n+3}],\\
    \cdots\quad \cdots \\
    C L^{N+ \frac{2\delta + (\delta-1)}{2^n}},\quad & \textrm{if   }\delta\in ]\frac{1}{2},\frac{3}{5}], \\
    C L^{N + \frac{\delta+\epsilon}{2^n}},\quad & \textrm{if   }\delta\in ]0,\frac{1}{2}].
  \end{cases}
\end{equation}
For each $\delta\in]0,\frac{1}{2}]$, and for $n$ sufficiently large, we get that the power of $L$ is less than $N+\epsilon_0$ for $\epsilon_0>0$
($\epsilon_0$ is the number appearing in \eqref{eq:Vcond});
while for each $\delta\in ]\frac{1}{2},1[$, there is some $m\in \mathbb{N}^+$ so that $\delta\in ]\frac{m+1}{m+3},\frac{m+2}{m+4}]$, thus after repeating the above process for
$m+n+1$ times, we get
\begin{equation*}
  \sup_{s\in [0,S_0]}\int_{|y|\leq L}|V(y,s)|^2 \,\mathrm{d}y\leq C L^{N+ \frac{2\delta + m(\delta-1)}{2^{n+1}}},\quad \textrm{for  }\delta\in \big]\frac{m+1}{m+3},\frac{m+2}{m+4}\big],
\end{equation*}
and for $n$ large enough, we infer that the power of $L$ is also less than $N+\epsilon_0$.
But this obviously contradicts with the estimation \eqref{eq:fact0} deduced from the condition \eqref{eq:Vcond},
which means there is no possibility to admit nontrivial velocity profiles in the case $-1<\alpha<-\delta$.

Now we prove \eqref{eq:Vconc}, and for this purpose, it suffices to prove the following inequality for all $-\delta\leq \alpha\leq -\epsilon_0$,
\begin{equation}\label{eq:V2est5}
  \frac{1}{L^{N-2\alpha}}\sup_{s\in[0,S_0]} \bigg(\int_{|y|\leq L}|V(y,s)|^2\,\mathrm{d}y\bigg) \gtrsim 1,\quad \forall L\gg1.
\end{equation}
Suppose \eqref{eq:V2est5} is not correct, then necessarily there exists a sequence of numbers $L_k\gg1$ such that
\begin{equation}
  \frac{1}{L_k^{N-2\alpha}}\sup_{s\in [0,S_0]}\bigg(\int_{|y|\leq L_k} |V(y,s)|^2\,\mathrm{d}y\bigg) \rightarrow 0,\quad \textrm{as}\quad L_k\rightarrow \infty.
\end{equation}
We also start from the local energy inequality $|J_2 -J_1|\leq C K_2$, and by letting $l_2=L_k\rightarrow\infty$ and $l_1=\lambda L$, we have
\begin{equation}\label{eq:V2key3}
\begin{split}
  \sup_{s\in[0,S_0]}\int_{|y|\leq L} |V(y,s)|^2\,\mathrm{d}y  \leq C L^{N-2\alpha} \sum_{k=0}^\infty
  \frac{1}{(\lambda^k L)^{N+1-2\alpha}}\int_0^{S_0}\int_{\lambda^k L\leq |y|\leq \lambda^{k+2} L}
  \big(|V|^3 + |P| |V|\big) \,\mathrm{d}y\mathrm{d}s,
\end{split}
\end{equation}
which is exactly the same as \eqref{eq:key1}.
Since we already have \eqref{eq:keyest2}, and by using \eqref{eq:estq2} in Lemma \ref{lem:pres2} with $b=N-2\alpha$, we have
\begin{equation*}
\begin{split}
  \sup_{s\in [0,S_0]}\int_{|y|\leq L} |V(y,s)|^2\,\mathrm{d}y  & \leq
  \begin{cases}
  \frac{C}{L} \sum_{k=0}^\infty \frac{1}{\lambda^{k (N-2\alpha + 1)}} (\lambda^k L)^{\max\{N-2\alpha +\delta,N-\alpha+1\}},
  \quad& \textrm{if}\;\; \alpha\neq -\frac{1}{2},\delta\neq \frac{1}{2}, \\
  \frac{C}{L} \sum_{k=0}^\infty \frac{1}{\lambda^{k (N-2\alpha + 1)}} (\lambda^k L)^{\frac{3}{2}}[\log_2 (\lambda^k L)],
  \quad &\textrm{if}\;\; \alpha= -\frac{1}{2},\delta= \frac{1}{2},
  \end{cases} \\
  & \leq
  \begin{cases}
  C L^{N-2\alpha+\delta-1}, \quad & \textrm{if   } \alpha\in [-\delta,\delta-1], \delta\in ]\frac{1}{2},1[, \\
  C L^{N+\frac{1}{2}+\epsilon},\quad & \textrm{if   }  \alpha=-\frac{1}{2}, \delta=\frac{1}{2}, \\
  C L^{N-\alpha},\quad & \textrm{if   }  \alpha\in [ \delta-1,-\frac{\epsilon_0}{2}], \delta\in [\epsilon_0,1- \epsilon_0],(\alpha,\delta)\neq (-\frac{1}{2},\frac{1}{2})
  \end{cases}
\end{split}
\end{equation*}
with $0<\epsilon\ll1/2$ a small number.
Using this improved estimate and Lemma \ref{lem:pres2} again, similarly as above we find
\begin{equation*}
\begin{split}
  \sup_{s\in [0,S_0]}\int_{|y|\leq L} |V(y,s)|^2\,\mathrm{d}y  & \leq
  \begin{cases}
    \frac{C}{L} \sum_{k=0}^\infty \frac{1}{\lambda^{k (N-2\alpha + 1)}} (\lambda^k L)^{N-2\alpha +2\delta-1},
    \quad &\textrm{if   }\alpha\in [-\delta,\frac{3}{2}(\delta-1)],\delta\in [\frac{3}{5},1[, \\
    \frac{C}{L} \sum_{k=0}^\infty \frac{1}{\lambda^{k (N-2\alpha + 1)}} (\lambda^k L)^{N-\alpha +\frac{\delta-1}{2} +1},
    \quad &\textrm{if   }\alpha\in [\frac{3}{2}(\delta-1),\delta-1],\delta\in ]\frac{1}{2},1[, \\
    \frac{C}{L} \sum_{k=0}^\infty \frac{1}{\lambda^{k (N-2\alpha + 1)}} (\lambda^k L)^{N+\frac{-\alpha+\epsilon}{2} +1},
    \quad &\textrm{if   }\alpha\in [\delta-1,-\frac{\epsilon_0}{2}],\delta\in [\epsilon_0,1-\epsilon_0], \\
  \end{cases} \\
  & \leq
  \begin{cases}
    L^{N-2\alpha +2\delta-2},\quad &\textrm{if   }\alpha\in [-\delta,\frac{3}{2}(\delta-1)],\delta\in [\frac{3}{5},1[, \\
    L^{N-\alpha +\frac{\delta-1}{2}},\quad &\textrm{if   }\alpha\in [\frac{3}{2}(\delta-1),\delta-1],\delta\in ]\frac{1}{2},1[, \\
    L^{N+\frac{-\alpha+\epsilon}{2}},\quad &\textrm{if   }\alpha\in [\delta-1,-\epsilon_0],\delta\in [\epsilon_0,1-\epsilon_0], \\
  \end{cases}
\end{split}
\end{equation*}
By repeating the above process for $n+1$ times leads to
\begin{equation}\label{eq:estim4}
\begin{split}
  \sup_{s\in [0,S_0]}\int_{|y|\leq L} |V|^2\,\mathrm{d}y
  \lesssim
  \begin{cases}
    L^{N-2\alpha +(n+1)(\delta-1)},\quad &\textrm{if   }\alpha\in [-\delta,\frac{n+2}{2}(\delta-1)],\delta\in [\frac{n+2}{n+4},1[, \\
    L^{N-\alpha +\frac{n}{2}(\delta-1)},\quad &\textrm{if   }\alpha\in [\frac{n+2}{2}(\delta-1),\frac{n+1}{2}(\delta-1)],\delta\in [\frac{n+1}{n+3},1[, \\
    \cdots\quad \cdots \\
    L^{N-\frac{\alpha}{2^{n-1}} +\frac{1}{2^n}(\delta-1)},\quad &\textrm{if   }\alpha\in [\frac{3}{2}(\delta-1),(\delta-1)],\delta\in ]\frac{1}{2},1[, \\
    L^{N+\frac{-\alpha+\epsilon}{2^n}},\quad &\textrm{if   }\alpha\in [\delta-1,-\epsilon_0],\delta\in [\epsilon_0,1-\epsilon_0]. \\
  \end{cases}
\end{split}
\end{equation}
From \eqref{eq:estim4}, we claim that for all $\alpha\in [-\delta,-\epsilon_0]$ and $\delta\in [\epsilon_0,1[$,
\begin{equation}\label{eq:claim}
  \sup_{s\in [0,S_0]}\int_{|y|\leq L} |V(y,s)|^2 \,\mathrm{d}y \lesssim L^{N+\epsilon_0},\quad \forall L\gg1.
\end{equation}
Indeed, we divide into three cases:
if $\delta\in [\epsilon_0,\frac{1}{2}[$, then the scope $[-\delta,-\epsilon_0]\subset [\delta-1,-\epsilon_0]$, and thus for $n$ large enough, we get \eqref{eq:claim} for all $-\delta\leq \alpha \leq -\epsilon_0$;
if $\delta\in [\frac{n+1}{n+3},\frac{n+2}{n+4}[$ for some $n\in \mathbb{N^+}$ and $\delta\leq 1-\epsilon_0$, then $-\delta>\frac{n+2}{2}(\delta-1)$,
and $\alpha\in[-\delta,-\epsilon_0]\subset [\frac{n+2}{2}(\delta-1),\frac{n+1}{2}(\delta-1)]\cup\cdots \cup [\frac{3}{2}(\delta-1),\delta-1]\cup [\delta-1,\epsilon_0]$,
thus after repeating the above process for $m+n+1$ times, we get for all $-\delta\leq \alpha\leq -\epsilon_0$,
\begin{equation}\label{eq:keyest3}
\begin{split}
  \sup_{s\in [0,S_0]}\int_{|y|\leq L} |V(y,s)|^2\,\mathrm{d}y
  & \lesssim
  \begin{cases}
    L^{N+ \frac{-\alpha}{2^{m}} +\frac{n}{2^{m+1}}(\delta-1)},\quad &\textrm{if   }\alpha\in [\frac{n+2}{2}(\delta-1),\frac{n+1}{2}(\delta-1)], \\
    \cdots\quad \cdots \\
    L^{N+\frac{-\alpha}{2^{m+n-1}} +\frac{1}{2^{m+n}}(\delta-1)},\quad &\textrm{if   }\alpha\in [\frac{3}{2}(\delta-1),(\delta-1)], \\
    L^{N+\frac{-\alpha+\epsilon}{2^{m+n}}},\quad &\textrm{if   }\alpha\in [\delta-1,-\epsilon_0],
  \end{cases} \\
  & \lesssim L^{N+\epsilon_0}, \qquad \forall L\gg1,
\end{split}
\end{equation}
where in the second line we have chosen $m$ large enough; finally, if $\delta\in [\frac{n+1}{n+3},\frac{n+2}{n+4}[$ for some $n\in \mathbb{N^+}$ and $\delta> 1-\epsilon_0$,
then $-\delta>\frac{n+2}{2}(\delta-1)$, $\delta-1>-\epsilon_0$, and
$\alpha\in[-\delta,-\epsilon_0]\subset [\frac{n+2}{2}(\delta-1),\frac{n+1}{2}(\delta-1)]\cup\cdots \cup [\frac{3}{2}(\delta-1),\delta-1]$,
we can obtain \eqref{eq:claim} similarly as getting \eqref{eq:keyest3} for all $-\delta\leq \alpha\leq-\epsilon_0$.
However, the estimate \eqref{eq:claim} clearly contradicts with \eqref{eq:fact0},
and thus the assumption \eqref{eq:V2est5} is not compatible with the condition \eqref{eq:Vcond}, and the desired estimate \eqref{eq:Vconc} is followed.

\subsection{Proof of Theorem \ref{thm:DSS2}-(2)}
Since $\alpha>-\frac{1}{2}$ and $\delta<\frac{1}{2}$ in \eqref{eq:Vcond2}, we have $A(s)\equiv 0$ in the representation formula of $P$ \eqref{eq:Pys0},
and we can use the better estimate \eqref{eq:estq3} instead of \eqref{eq:estq2} in the main proof.
First we also have \eqref{eq:keyest2} for all $\alpha>-1$, and in combination with the condition \eqref{eq:Vcond2}, we infer that the only possible range of $\alpha$
to admit nontrivial velocity profiles is $\{\alpha:-1<\alpha \leq 0\}$, since we need that
\begin{equation}\label{eq:fact3}
L^N\lesssim \sup_{s\in [0,S_0]}\int_{|y|\leq L} |V(y,s)|^2\,\mathrm{d}y  \lesssim L^{N-2\alpha},\quad \forall L\gg1.
\end{equation}
Next we consider the case $-1<\alpha<-\delta$. Similarly as above, we also begin with \eqref{eq:key1}, and by virtue of \eqref{eq:Ves-rou} and \eqref{eq:estq3} in Lemma \ref{lem:pres2} below, we get
\begin{equation*}
  \int_0^{S_0}\int_{|y|\leq \lambda^{k+2} L} |P(y,s)| |V(y,s)|\,\mathrm{d}y \mathrm{d}s \lesssim (\lambda^k L)^{N+3\delta},
\end{equation*}
and
\begin{equation*}
  \sup_{s\in [0,S_0]}\int_{|y|\leq L}|V(y,s)|^2\,\mathrm{d}y\leq \frac{C}{L}\sum_{k=0}^\infty \frac{1}{\lambda^{k(N-2\alpha+1)}} (\lambda^k L)^{N+3\delta}\leq C L^{N+3\delta-1}.
\end{equation*}
We can repeatedly use this process to show that
\begin{equation*}
  \sup_{s\in[0,S_0]}\int_{|y|\leq L} |V(y,s)|^2\,\mathrm{d}y \leq C L^{N+2\delta - (n+1)(1-\delta)},
\end{equation*}
as long as $N+2\delta -n(1-\delta) \geq N-2\delta$, that is, $n\leq \frac{4\delta}{1-\delta}$. Set $n_0 = [\frac{4\delta}{1-\delta}]$, then we obtain
\begin{equation*}
  \sup_{s\in [0,S_0]}\int_{|y|\leq L} |V(y,s)|^2\,\mathrm{d}y \leq C L^{N+2\delta - (n_0+1)(1-\delta)}\leq C L^{N-2\delta},
\end{equation*}
which clearly contradicts with the lower bound in \eqref{eq:fact3},
and means that the case $-1<\alpha<\delta$ is not compatible.

In the end for the nontrivial velocity profiles corresponding to each $\alpha\in [-\delta,-\epsilon_0]$, we prove \eqref{eq:Vconc}, and it suffices to prove \eqref{eq:V2est5} for all $\alpha$ in this range.
Similarly as above, we begin with \eqref{eq:V2key3} to get
\begin{equation*}
  \sup_{s\in [0,S_0]}\int_{|y|\leq L}|V(y,s)|^2\,\mathrm{d}y \leq \frac{C}{L}\sum_{k=0}^\infty \frac{1}{\lambda^{k(N-2\alpha+1)}} (\lambda^k L)^{N-2\alpha +\delta}\leq C L^{N-2\alpha+\delta-1}.
\end{equation*}
By iteration, we can show that, as long as $N+2\alpha-n(1-\delta)\geq N-2\delta$,
\begin{equation*}
  \sup_{s\in [0,S_0]}\int_{|y|\leq L} |V(y,s)|^2\,\mathrm{d}y \leq C L^{N-2\alpha- (n+1)(1-\delta)}.
\end{equation*}
Set $n_0'=[\frac{2\alpha+2\delta}{1-\delta}]$, thus we find
\begin{equation*}
  \sup_{s\in [0,S_0]}\int_{|y|\leq L} |V(y,s)|^2\, \mathrm{d}y \leq C L^{N-2\alpha -(n'_0+1)(1-\delta)} \leq C L^{N-2\delta},
\end{equation*}
which contradicts with the lower bound in \eqref{eq:fact3}, and thus proves \eqref{eq:V2est5} and \eqref{eq:Vconc} for every $-\delta\leq\alpha\leq 0$.

\section{Auxiliary lemmas: estimation of the pressure profile}\label{sec:lem}

\begin{lemma}\label{lem:pres1}
  Suppose that $V\in C^1_s C^1_{y,\mathrm{loc}}(\mathbb{R}^{N+1})$ is a locally periodic-in-$s$
vector field with period $S_0>0$, which additionally satisfies that for every $L\gg1$, $2<p<\infty$ and $2\leq r\leq\infty$,
\begin{equation*}
  \Big\|\Big(\int_{|y|\leq L}|V(y,s)|^p\,\mathrm{d}y\Big)^{\frac{1}{p}}\Big\|_{L^r([0,S_0])}\lesssim L^{\frac{a}{p}},\quad \textrm{with}\quad
  0\leq a <N.
\end{equation*}
Let $P(y,s)$ be a scalar-valued function defined from $V$ by
\begin{equation}\label{eq:Pys2}
  P(y,s):= c_0 |V(y,s)|^2 + \mathrm{p.v.} \int_{\mathbb{R}^N} K_{ij}(y-z) V_i(z,s) V_j(z,s)\,\mathrm{d}z
\end{equation}
with $c_0\in\mathbb{R}$ and $K_{ij}(z)$ ($i,j=1,\cdots,N$) some Calder\'on-Zygmund kernel, then we have
\begin{equation}\label{eq:Pest}
  \Big\|\Big(\int_{|y|\leq L} |P(y,s)|^{\frac{p}{2}}\,\mathrm{d}y\Big)^{\frac{2}{p}}\Big\|_{L^{\frac{r}{2}}([0,S_0])} \lesssim L^{\frac{2a}{p}}.
\end{equation}

\end{lemma}

\begin{proof}[Proof of Lemma \ref{lem:pres1}]
We only suffice to treat the integral term in the expression formula \eqref{eq:Pys2}, denoting by $\tilde{P}(y,s)$,
and we use the following decomposition
\begin{equation*}
\begin{split}
  \tilde{P}(y,s)& \,= \mathrm{p.v.} \int_{|z|\leq 2L} K_{ij}(y-z) V_i(z,s) V_j(z,s)\,\mathrm{d}z
  + \int_{|z|\geq 2L} K_{ij}(y-z) V_i(z,s) V_j(z,s)\,\mathrm{d}z \\
  & \,:= \tilde{P}_{1,L}(y,s) + \tilde{P}_{2,L}(y,s).
\end{split}
\end{equation*}
By the Calder\'on-Zygmund theorem, we first see that
\begin{equation*}
  \Big\|\Big(\int_{|y|\leq L} |\tilde{P}_{1,L}(y,s)|^{\frac{p}{2}}\,\mathrm{d}y\Big)^{\frac{2}{p}}\Big\|_{L^{r/2}_s}
  \lesssim \Big\|\Big(\int_{|y|\leq 2L} |V(y,s)|^p\,\mathrm{d}y\Big)^{1/p} \Big\|_{L^r_s}^2 \lesssim L^{\frac{2a}{p}}.
\end{equation*}
For $\tilde{P}_{2,L}$, by the dyadic decomposition, Minkowski's inequality and H\"older's inequality we have
\begin{equation*}
\begin{split}
  \Big\|\Big(\int_{|y|\leq L} |\tilde{P}_{2,L}(y,s)|^{\frac{p}{2}}\,\mathrm{d}y\Big)^{\frac{2}{p}}\Big\|_{L^{r/2}_s}
  & \lesssim \Big\|\Big(\int_{|y|\leq L}\Big(\sum_{k=1}^\infty\int_{2^k L\leq |z|\leq 2^{k+1}L}\frac{1}{|y-z|^N}
  |V(z,s)|^2\,\mathrm{d}z\Big)^{\frac{p}{2}}\,\mathrm{d}y\Big)^{\frac{2}{p}}\Big\|_{L^{r/2}_s} \\
  & \lesssim \sum_{k=1}^\infty \Big\|\Big(\int_{|y|\leq L}\Big(\int_{|z|\sim 2^k L}\frac{1}{|z|^N}
  |V(z,s)|^2\,\mathrm{d}z\Big)^{\frac{p}{2}}\,\mathrm{d}y\Big)^{\frac{2}{p}}\Big\|_{L^{r/2}_s} \\
  & \lesssim L^{\frac{2N}{p}} \sum_{k=1}^\infty (2^k L)^{-N}\Big\|\int_{|z|\sim 2^k L}
  |V(z,s)|^2\,\mathrm{d}z\Big\|_{L^{r/2}_s} \\
  & \lesssim L^{\frac{2N}{p}} \sum_{k=1}^\infty (2^k L)^{-2N/p } \Big\|\Big(\int_{|z|\sim 2^k L}
  |V(z,s)|^p\,\mathrm{d}z\Big)^{\frac{2}{p}}\Big\|_{L^{r/2}_s} \\
  & \lesssim L^{\frac{2N}{p}}\sum_{k=1}^\infty(2^k L)^{-\frac{2(N-a)}{p}} \lesssim L^{\frac{2a}{p}}.
\end{split}
\end{equation*}
Hence gathering the above estimates yields \eqref{eq:Pest}.

\end{proof}

\begin{lemma}\label{lem:pres2}
  Assume that $V\in C^1_s C^3_{y,\mathrm{loc}}(\mathbb{R}^{N+1};\mathbb{R}^N)$ is a periodic-in-$s$ vector field with period $S_0$, and additionally
$V$ satisfies that
\begin{equation}\label{eq:lemasum}
\begin{split}
  \sup_{s\in [0,S_0]} |V(y,s)| \lesssim |y|^\delta,\;\;\forall |y|\geq M, \quad&\textrm{with}\quad 0\leq \delta<1\quad\textrm{and}\\
  \sup_{s\in[0,S_0]} \int_{|y|\leq L} |V(y,s)|^2\,\mathrm{d}y \lesssim L^b,\;\; \forall L\geq M,\quad &\textrm{with}\quad 0\leq b\leq  N+2\delta,
\end{split}
\end{equation}
with $M>0$ a fixed number.
Let $Q(y,s)$ be a scalar field defined from $V(y,s)$ by that
\begin{equation}
\begin{split}
  Q(y,s)= &\, c_0 |V(y,s)|^2 + A(s)\cdot y + p.v. \int_{\mathbb{R}^N} K_{ij}(y-z) V_i(z,s) V_j(z,s)\,\mathrm{d}z\, + \\
  & \,+
  \begin{cases}
  -\int_{|z|\geq M} K_{ij}(z) V_i(z,s) V_j(z,s)\,\mathrm{d}z, \quad& \textrm{if  }\delta\in [0,1/2[, \\
  -\int_{|z|\geq M} \big(K_{ij}(z)+y\cdot\nabla K_{ij}(z)\big) V_i(z,s) V_j(z,s)\,\mathrm{d}z, \quad & \textrm{if  }\delta\in[1/2,1[,
  \end{cases}
\end{split}
\end{equation}
where $c_0\in\mathbb{R}$, $A(s)\in C(\mathbb{R};\mathbb{R}^N)$ is a periodic-in-$s$ function with period $S_0$ and $K_{ij}(z)$ ($i,j=1,\cdots,N$) is a Calder\'on-Zygmund type kernel, then we have
\begin{equation}\label{eq:estq2}
  \int_0^{S_0}\int_{|y|\leq L} |Q(y,s)| |V(y,s)| \,\mathrm{d}y \mathrm{d}s \lesssim
  \begin{cases}
    L^{b+\delta} +  L^{\frac{N+b}{2}+1},\quad & \textrm{if}\;\; (b,\delta)\neq (N+1,\frac{1}{2}),\\
    L^{\frac{N+3}{2}}[\log_2 L], \quad &\textrm{if}\;\; (b,\delta)=(N+1,\frac{1}{2}).
  \end{cases}
\end{equation}
In particular, if $\delta\in [0,\frac{1}{2}[$ in \eqref{eq:lemasum} and $A(s)\equiv 0$, we also have
\begin{equation}\label{eq:estq3}
  \int_0^{S_0}\int_{|y|\leq L} |Q(y,s)| |V(y,s)| \,\mathrm{d}y \mathrm{d}s \lesssim
  \begin{cases}
    L^{b+\delta},\quad & \textrm{if}\;\; b\geq N-2\delta,(b,\delta)\neq (N,0), \\
    L^N [\log_2 L],\quad & \textrm{if}\;\; (b,\delta)=(N,0), \\
    L^{\frac{N+b}{2}} ,\quad & \textrm{if}\;\; b\leq N-2\delta,(b,\delta)\neq (N,0). \\
  \end{cases}
\end{equation}
\end{lemma}

\begin{proof}[Proof of Lemma \ref{lem:pres2}]

We decompose $Q(y,s)$ as
\begin{equation}\label{eq:Qy-dec}
  Q(y,s)= c_0 |V(y,s)|^2 + Q_{1,L}(y, s) + Q_{2,L}(y,s) + Q_{3,L} (y, s) + Q_{4,L}(y,s),
\end{equation}
where
\begin{equation*}
\begin{split}
  & Q_{1,L}(y, s)= A(s)\cdot y, \quad\quad Q_{2,L}(y, s) = \textrm{p.v.} \int_{|y|\leq 2L} K_{ij}(y-z) V_i(z,s) V_j(z,s) \,\mathrm{d}z, \\
  & Q_{3,L}(y, s)=
  \begin{cases}
    \int_{|z|\geq 2 L} \big(K_{ij}(y-z)-K_{ij}(z)\big) V_i(z,s) V_j(z,s)\,\mathrm{d}z, \quad &\textrm{if}\;\; \delta\in [0,\frac{1}{2}[\\
    \int_{|z|\geq 2 L} \big(K_{ij}(y-z)-K_{ij}(z)-y\cdot\nabla K_{ij}(z)\big) V_i(z,s) V_j(z,s)\,\mathrm{d}z, \quad &\textrm{if}\;\; \delta\in [\frac{1}{2},1[,
  \end{cases}
  \\
  & Q_{4,L}(y, s)=
  \begin{cases}
    -\int_{M\leq |z|\leq 2L} K_{ij}(z) V_i(z,s) V_j(z,s)\,\mathrm{d}z, \quad &\textrm{if}\;\; \delta\in[0,\frac{1}{2}[, \\
    -\int_{M\leq |z|\leq 2L} \big(K_{ij}(z) + y\cdot\nabla K_{ij}(z)\big) V_i(z,s) V_j(z,s)\,\mathrm{d}z, \quad &\textrm{if}\;\; \delta\in[\frac{1}{2},1[.
  \end{cases}
\end{split}
\end{equation*}
From \eqref{eq:lemasum}, we first directly have
$$\int_0^{S_0}\int_{|y|\leq L}|V(y,s)|^3\,\mathrm{d}y \mathrm{d}s \lesssim L^{b+\delta},$$
and
\begin{equation*}
\begin{split}
  \int_0^{S_0} \int_{|y|\leq L} |Q_{1,L}(y,s)| |V(y,s)|\,\mathrm{d}y \mathrm{d}s &\leq \Big(\sup_{s\in [0,S_0]}|A(s)|\Big) L^{N/2+1} \bigg(\int_0^{S_0}\int_{|y|\leq L} |V(y,s)|^2\,\mathrm{d}y\mathrm{d}s\bigg)^{1/2} \\
  & \lesssim  L^{\frac{N+b}{2}+ 1}.
\end{split}
\end{equation*}
For the term involving $Q_{2,L}(y,s)$,
by the H\"older inequality and Calder\'on-Zygmund theorem, we get
\begin{equation*}
\begin{split}
  \int_0^{S_0}\int_{|y|\leq L}|Q_{2,L}(y,s)| |V(y,s)|\,\mathrm{d}y\mathrm{d}s & \leq \Big(\int_0^{S_0}\int_{|y|\leq L}|Q_{2,L}(y,s)|^{\frac{3}{2}}\,\mathrm{d}y\mathrm{d}s\Big)^{\frac{2}{3}}
  \Big( \int_0^{S_0}\int_{|y|\leq L} |V(y,s)|^3\,\mathrm{d}y\mathrm{d}s\Big)^{\frac{1}{3}} \\
  & \lesssim \int_0^{S_0}\int_{|y|\leq 2L} |V(y,s)|^3\,\mathrm{d}y\mathrm{d}s \lesssim L^{b+\delta}.
\end{split}
\end{equation*}
For the term containing $Q_{3,L}(y,s)$, using the support property and the dyadic decomposition again, we infer that if $\delta\in [0,1/2[$,
\begin{equation*}
\begin{split}
  \int_0^{S_0}\int_{|y|\leq L} |Q_{3,L}(y,s)| |V(y,s)|\,\mathrm{d}y\mathrm{d}s & \lesssim L^{N+\delta} \int_0^{S_0}\Big(\sup_{|y|\leq L} |Q_{3,L}(y,s)|\Big)\,\mathrm{d}s \\
  & \lesssim L^{N+\delta} \sup_{|y|\leq L} \bigg(\sum_{k=1}^\infty \int_0^{S_0}\int_{2^k L\leq |z|\leq 2^{k+1}L}
  \frac{|y|}{|z|^{N+1}} |V(z,s)|^2\,\mathrm{d}z\mathrm{d}s\bigg) \\
  & \lesssim L^{N+\delta+1} \sum_{k=1}^\infty \frac{1}{(2^k L)^{N+1}} \int_0^{S_0}\int_{|z|\sim 2^k L} |V(z,s)|^2\,\mathrm{d}z\,\mathrm{d}s \\
  & \lesssim L^{N+\delta+1}\sum_{k=1}^\infty (2^k L)^{b-N-1}\lesssim L^{b+\delta},
\end{split}
\end{equation*}
and if $\delta\in [1/2,1[$,
\begin{equation*}
\begin{split}
  \int_0^{S_0}\int_{|y|\leq L} |Q_{3,L}(y,s)| |V(y,s)|\,\mathrm{d}y\mathrm{d}s & \lesssim L^{N+\delta} \int_0^{S_0}\Big(\sup_{|y|\leq L} |Q_{3,L}(y,s)|\Big)\,\mathrm{d}s \\
  & \lesssim L^{N+\delta} \sup_{|y|\leq L} \bigg(\sum_{k=1}^\infty \int_0^{S_0}\int_{2^k L\leq |z|\leq 2^{k+1}L}
  \frac{|y|^2}{|z|^{N+2}} |V(z,s)|^2\,\mathrm{d}z \mathrm{d}s\bigg) \\
  & \lesssim L^{N+\delta+2} \sum_{k=1}^\infty \frac{1}{(2^k L)^{N+2}} \int_0^{S_0}\int_{|z|\sim 2^k L} |V(z,s)|^2\,\mathrm{d}z\mathrm{d}s \\
  & \lesssim L^{N+\delta+2} \sum_{k=1}^\infty (2^k L)^{b-N-2}\lesssim L^{b+\delta}.
\end{split}
\end{equation*}
For the last term,  thanks to H\"older's inequality and the dyadic decomposition, we deduce that if $\delta\in [0,\frac{1}{2}[$,
\begin{equation*}
\begin{split}
  \int_0^{S_0}\int_{|y|\leq L} |Q_{4,L}(y,s)| |V(y,s)|\,\mathrm{d}y\mathrm{d}s & \lesssim L^{N/2} \Big(\int_0^{S_0}\int_{|y|\leq L}|V(y,s)|^2\,\mathrm{d}y\mathrm{d}s\Big)^{\frac{1}{2}}
  \Big(\sup_{s\in[0,S_0];|y|\leq L}|Q_{4,L}(y,s)|\Big) \\
  & \lesssim L^{\frac{N+b}{2}} \sup_{s\in[0,S_0]}\bigg(\sum_{k=-1}^{[\log_2 \frac{L}{M}]}\int_{\frac{L}{2^{k+1}}\leq |z|\leq \frac{L}{2^k}} \frac{1}{|z|^N} |V(z,s)|^2\,\mathrm{d}z\bigg) \\
  & \lesssim L^{\frac{N+b}{2}} \sum_{k=-1}^{[\log_2 \frac{L}{M}]} \Big(\frac{L}{2^k}\Big)^{-N+b} \lesssim
  \begin{cases}
    L^{\frac{3b-N}{2}},\quad &\textrm{if}\;\; b>N, \\
    L^N [\log_2 L],\quad &\textrm{if}\;\; b=N, \\
    L^{\frac{N+b}{2}},\quad &\textrm{if}\;\; b<N,
  \end{cases} \\
  & \lesssim
  \begin{cases}
    L^{b+\delta},\quad & \textrm{if}\;\; b\geq N,(b,\delta)\neq (N,0), \\
    L^N [\log_2 L],\quad & \textrm{if}\;\; (b,\delta)=(N,0), \\
    L^{\frac{N+b}{2}} ,\quad & \textrm{if}\;\; b<N, \\
  \end{cases}
\end{split}
\end{equation*}
and if $\delta\in[\frac{1}{2},1[$,
\begin{equation*}
\begin{split}
  & \int_0^{S_0}\int_{|y|\leq L} |Q_{4,L}(y,s)| |V(y,s)|\,\mathrm{d}y\mathrm{d}s \\
  \lesssim\, & L^{N/2} \Big(\int_0^{S_0}\int_{|y|\leq L}|V(y,s)|^2\,\mathrm{d}y\mathrm{d}s \Big)^{1/2} \Big(\sup_{s\in [0,S_0];|y|\leq L}|Q_{4,L}(y,s)|\Big) \\
  \lesssim\, & L^{\frac{N+b}{2}} \sup_{s\in[0,S_0]}\bigg(\sum_{k=-1}^{[\log_2 \frac{L}{M}]}\int_{\frac{L}{2^{k+1}}\leq |z|\leq \frac{L}{2^k}} \Big(\frac{1}{|z|^N} + \frac{L}{|z|^{N+1}}\Big) |V(z,s)|^2\,\mathrm{d}z\bigg) \\
  \lesssim\, & L^{\frac{N+b}{2}+1} \sum_{k=-1}^{[\log_2 \frac{L}{M}]} \Big(\frac{L}{2^k}\Big)^{-N-1+b} \lesssim
  \begin{cases}
    L^{\frac{3b-N}{2}},\quad &\textrm{if}\;\; b>N+1, \\
    L^{N+\frac{3}{2}} [\log_2 L],\quad &\textrm{if}\;\; b=N+1, \\
    L^{\frac{N+b}{2}+1},\quad &\textrm{if}\;\; b<N+1,
  \end{cases}
  \\ \lesssim\, &
  \begin{cases}
    L^{b+\delta},\quad & \textrm{if}\;\; b\geq N+1,\,(b,\delta)\neq (N+1,\frac{1}{2}), \\
    L^{N+\frac{3}{2}}[\log_2 L], \quad &\textrm{if}\;\; (b,\delta)= (N+1,\frac{1}{2}), \\
    L^{\frac{N+b}{2}+1},\quad & \textrm{if}\;\; b<N+1.
  \end{cases}
\end{split}
\end{equation*}
Therefore, collecting the above estimates leads to the desired estimates \eqref{eq:estq2} and \eqref{eq:estq3}.
\end{proof}

\textbf{Acknowledgements.}
The author was partially supported by NSFC grant 11401027 and a special fund from the Laboratory of Mathematics and Complex Systems, Ministry of Education.

\end{document}